\documentclass[11pt,reqno]{amsart}
\usepackage{amsmath}
\usepackage{amsthm}
\usepackage{graphicx}
\usepackage{amsfonts}
\usepackage{amssymb}
\usepackage{hyperref}
\usepackage{comment}
\usepackage{color}
\usepackage{bm}
\usepackage[margin=1.1in]{geometry}
\usepackage{enumerate}
\usepackage{mathrsfs}

\numberwithin{equation}{section}

\newtheorem{theorem}{Theorem}[section]
\newtheorem{lemma}[theorem]{Lemma}

\newtheorem{corollary}[theorem]{Corollary}

\theoremstyle{definition}
\newtheorem{example}[theorem]{Example}

\newtheorem{remark}[theorem]{Remark}

\newtheorem{innercustomthm}{Assumption}
\newenvironment{assumption}[1]
  {\renewcommand\theinnercustomthm{#1}\innercustomthm}
  {\endinnercustomthm}

\DeclareMathOperator*{\esssup}{ess\,sup}

\def\E{{\mathbb E}}
\def\R{{\mathbb R}}
\def\N{{\mathbb N}}
\def\Z{{\mathbb Z}}
\def\P{{\mathcal P}}

\def\W{{\mathcal W}}
\def\T{{\mathbb T}}
\def\div{\mathrm{div}}

\def\LSI{\eta}
\def\approx1{\zeta}

\title[Sharp uniform-in-time propagation of chaos]{Sharp uniform-in-time propagation of chaos}

\author{Daniel Lacker and Luc Le Flem}
\address{Department of Industrial Engineering \& Operations Research, Columbia University}
\email{daniel.lacker@columbia.edu}
\email{ll3240@columbia.edu}
\thanks{D.L.\ is partially supported by the AFOSR Grant FA9550-19-1-0291 and the NSF CAREER award DMS-2045328. L.L.\  is partially supported by the same NSF award.}

\begin{document}

\begin{abstract}
We prove the optimal rate of quantitative propagation of chaos, uniformly in time, for interacting diffusions. Our main examples are interactions governed by convex potentials and models on the torus with small interactions. We show that the distance between the $k$-particle marginal of the $n$-particle system and its limiting product measure is $O((k/n)^2)$, uniformly in time, with distance measured either by relative entropy, squared quadratic Wasserstein metric, or squared total variation. Our proof is based on an analysis of relative entropy through the BBGKY hierarchy, adapting prior work of the first author to the time-uniform case  by means of log-Sobolev inequalities.
\end{abstract}

\maketitle

\section{Introduction}

This paper studies interacting diffusion processes in $\R^d$ of the form
\begin{align}
d X_t^i = \bigg(b_0(t, X_t^i) + \frac{1}{n-1} \sum_{j=1,j\neq i}^n b(t,X_t^i, X_t^j)\bigg)dt + \sigma d W_t^i, \qquad i =1, \ldots, n,
\label{eq:intro-interacting-particle-system}
\end{align}
where $W^1, \ldots, W^n$ are independent $d$-dimensional Brownian motions and $\sigma > 0$ is scalar. 
Precise assumptions are deferred to Section \ref{se:main-results}.
Typically, as $n\to\infty$, the limiting behavior of this $n$-particle system is described in terms of the McKean-Vlasov equation
\begin{align}
d X_t = \left(b_0(t, X_t) + \langle \mu_t, b(t,X_t,\cdot)\rangle\right)dt + \sigma d W_t, \qquad \mu_t = \text{Law}(X_t),
\label{eq:intro-McKean-Vlasov}
\end{align}
or the associated nonlinear Fokker-Planck equation
\begin{align}
\partial_t\mu_t(x) = -\mathrm{div}_x\big((b_0(t,x) + \langle \mu_t,b(t,x,\cdot)\rangle ) \mu_t(x)\big) + (\sigma^2/2)\Delta_x\mu_t(x), \label{eq:intro-McKean-Vlasov-PDE}
\end{align}
where we abbreviate $\langle \mu_t,b(t,x,\cdot)\rangle = \int_{\R^d} b(t,x,y)\,\mu_t(dy)$.
The large-$n$ limit is formalized by the notion of \emph{propagation of chaos} originating in \cite{kac1956foundations,mckean1967propagation}. Suppose for this introduction that the above equations are well-posed, with the law of $(X^1_t,\ldots,X^n_t)$ being exchangeable at $t=0$ and thus also at each $t > 0$. Let $P^{n,k}_t=\mathrm{Law}(X^1_t,\ldots,X^k_t)$ denote the $k$-particle marginal. The problem of propagation of chaos is to show that $P^{n,k}_t \to \mu^{\otimes k}_t$ weakly as $n \to \infty$, for each $k\in \N$ and $t > 0$, where $\mu^{\otimes k}_t$ denotes the $k$-fold product measure (assuming that the same is true at $t=0$, which is trivially the case if the initial positions are i.i.d.). This is well known to be equivalent to the statement that the empirical measure
\[
L^n_t := \frac{1}{n}\sum_{i=1}^n\delta_{X^i_t}
\]
converges weakly in probability to $\mu_t$ as $n\to\infty$, for each $t > 0$. 

Over several decades of intensive research, propagation of chaos has been rigorously justified in a wide variety of contexts which we do not attempt to list here, instead referring to the comprehensive recent surveys \cite{ChaintronDiezI,ChaintronDiezII} as well as \cite{jabinwang-review}.
While the earliest results were qualitative in nature, recent work has succeeded in \emph{quantifying} propagation of chaos in various ways.
Let us first discuss some typical quantitative results which hold on bounded time intervals.
The classical synchronous coupling argument, going back to \cite{sznitman1991topics} in the case where $(b_0,b)$ are Lipschitz, gives a convergence rate for $\E[\W_2^2(L^n_t,\mu_t)]$ of the same order as if $L^n_t$ were an empirical measure of i.i.d.\ random variables, where $\W_2$ denotes the quadratic Wasserstein distance.
While this estimate deteriorates with the dimension, the same coupling argument yields the dimension-free $\W_2^2(P^{n,k}_t,\mu^{\otimes k}_t) = O(k/n)$, and we note also the recent work \cite{han2021class} that gives dimension-free estimates for $\E[d^2(L^n_t,\mu_t)]$ for certain metrics $d$ dominated by $\W_2$.
A recently popular approach is to estimate the relative entropy $H(P^{n,n}_t\,|\,\mu^{\otimes n}_t)$ of the full $n$-particle system; if it is shown to be bounded with respect to $n$, then the subadditivity of entropy yields $H(P^{n,k}_t\,|\,\mu^{\otimes k}_t) = O(k/n)$, or by Pinsker's inequality $\|P^{n,k}_t-\mu_t^{\otimes k}\|_{\mathrm{TV}}^2 =O(k/n)$. Relative entropy methods have gained popularity in part due to their ability to handle models with singular interactions relevant in physics, which were inaccessible by other methods \cite{bresch2020mean,jabin-wang-bounded,jabin-wang-W1inf}.

The recent paper \cite{LackerFoundation} of the first named author showed that the previous three estimates can actually be improved from $O(k/n)$ to $O((k/n)^2)$ in many cases, with a simple Gaussian example showing that the latter cannot be further improved, thus obtaining the sharp rate of propagation of chaos.
The broad class of models treated in \cite{LackerFoundation} includes bounded or uniformly continuous interaction functions, among others, but nothing too singular.
The novelty of the approach of \cite{LackerFoundation}, which is also based on relative entropy estimates, is its \emph{local} character, estimating $H^k_t:=H(P^{n,k}_t\,|\,\mu^{\otimes k}_t)$ hierarchically in terms of $H^{k+1}_t$ for each $k < n$ rather than estimating the \emph{global} quantity $H^n_t$. Interestingly, both of the estimates $H^n_t= O(1)$ and $H^k_t=O((k/n)^2)$ are typically sharp, which means that something is lost in using subadditivity to pass from global to local bounds.
Note by Pinsker's inequality that the latter estimate implies that $P^{n,k}_t$ has total variation distance at most $O(k/n)$ from the i.i.d.\ measure $\mu^{\otimes k}_t$, which is interesting to contrast with the well known result of Diaconis-Freedman \cite[Theorem 13]{diaconis1980finite} about finite exchangeable measures on general spaces; the Diaconis-Freedman bound, which is sharp at its level of generality, implies merely that the total variation distance between $P^{n,k}_t$ and \emph{some mixture} of i.i.d.'s is $O(k^2/n)$.

In this paper, we adapt the $O((k/n)^2)$  estimate of \cite{LackerFoundation} to be \emph{uniform in time}, under an additional assumption of a log-Sobolev inequality (LSI) for $\mu_t$, uniformly in $t$. 
Before discussing the results in more detail, let us briefly review some typical prior uniform-in-time results, which were all obtained by \emph{global} methods.
First, it is an important and well known fact that McKean-Vlasov equations may admit multiple invariant measures even when the $n$-particle system is uniquely ergodic for each $n$, which precludes a uniform-in-time convergence; see \cite{dawson1983critical,herrmann2010non} for examples. 
Early results on uniform-in-time propagation of chaos appeared in \cite{Malrieu2001,Malrieu2003}, based on the synchronous coupling approach combined with a uniform-in-$n$ LSI for the $n$-particle system, in the case of uniformly convex potentials:
\begin{align}
b_0(t,x)= -\nabla U(x), \quad b(t,x,y)=-\nabla W(x-y), \quad \nabla^2 U \succeq \alpha I, \ \nabla^2 W \succeq 0, \ \alpha > 0. \label{intro:potentials}
\end{align}
In particular, it is shown in \cite{Malrieu2001,Malrieu2003} that $\W_2^2(P^{n,k}_t,\mu^{\otimes k}_t)$  and $H(P^{n,k}_t \,|\, \mu^{\otimes k}_t)$ are both $O(k/n)$, uniformly in time.
This approach was extended in \cite{Malrieu2008} to relax the uniform convexity to mere convexity-at-infinity.
More recent developments have managed to obtain  similar rates under weaker convexity requirements, as long as the interaction strength is small or the temperature $\sigma$ sufficiently large \cite{Durmus2020,Moral2019,salem2020gradient,Salhi2017}.
A different but still global method is adopted in the recent paper \cite{Delarue2021} dealing with models on the torus $\T^d$ via analysis of a PDE (``master equation") set on $\P(\T^d)$, obtaining estimates like $\sup_{t \ge 0}|\E F(L^n_t) - F(\mu_t)|=O(1/n)$ for sufficiently smooth functionals $F$.
The recent papers \cite{guillin2021uniform,rosenzweig2021global} treat specific physically relevant models with singular interactions, again working globally and thus not obtaining our optimal rate. 

A noteworthy source of recent applications beyond physics is the analysis of large neural networks trained by stochastic gradient descent \cite{Bach2018,Mei2018,Sirignano2020}, which is well approximated by particle systems of the form \eqref{eq:intro-interacting-particle-system} and for which long-time behavior is particularly important.
On a more mathematical note, uniform-in-time propagation of chaos can also be used to derive the rate of convergence of $\mu_t$ to equilibrium as $t \to\infty$, as in \cite{Malrieu2008,Malrieu2001,Malrieu2003}.

The main contribution of this paper is to identify the sharp rate of propagation of chaos, along with a methodology that is amenable to further development. 
On purely qualitative grounds, though, we do not claim to cover any notable models for which propagation of chaos was not already known. As in \cite{LackerFoundation}, we state general results (Theorem \ref{th:main-theorem} and \ref{th:reverse}) under somewhat abstract assumptions, which highlight the key  a priori estimates one needs in order to implement our method. Most important among them are that $\mu_t$ obeys a transport-type inequality, which already appeared in \cite{LackerFoundation}, as well as a LSI, with constants uniform in time; the latter was not needed in the finite-time results of \cite{LackerFoundation}.
The main examples that we cover are convex-at-infinity  potentials as in \eqref{intro:potentials} with $\nabla^2W$ also being bounded from above, and certain models on the torus with $b_0 \equiv 0$ and $b(t,x,y)=K(x-y)$ for a vector field $K$ of small divergence.

The role of the LSI in our work warrants further discussion. 
Starting from \cite{Malrieu2001}, a common ingredient in a proof of uniform-in-time propagation of chaos is a uniform-in-$n$ LSI for the invariant measure of the $n$-particle system. 
The remarkable recent result of \cite[Theorem 3.7]{delgadino2021phase} shows that this uniform-in-$n$ LSI implies uniform-in-time propagation of chaos, with minimal additional assumptions.
See also the recent paper \cite{GuillinKinetic2021} on quantitative uniform-in-time propagation of chaos for kinetic models, which applies recent uniform LSIs for invariant measures developed in \cite{guillin2019uniform}.
Based similarly on a LSI for $P^{n,n}_t$ uniform in $n$ and $t$ (plus other non-trivial assumptions),  we obtain in Theorem \ref{th:reverse} an estimate on the reversed entropy $H(\mu^{\otimes k}_t\,|\,P^{n,k}_t) = O((k/n)^2)$, and note that this  optimal  rate $O((k/n)^2)$ cannot be recovered from the aforementioned results.
On the other hand, our first Theorem \ref{th:main-theorem} is perhaps more novel in that it proves the optimal rate $H(P^{n,k}_t\,|\,\mu^{\otimes k}_t)=O((k/n)^2)$ under essentially no assumptions on  the $n$-particle system itself, but instead we need a sufficiently high temperature and a LSI for the measure $\mu_t$ uniformly in $t$. See Theorem \ref{th:main-theorem} for the general result, and Corollary \ref{cor:convex-main} and \ref{co:torus} for the examples of convex interactions and small interactions on the torus, respectively. 
Among prior work, the closest to our approach appears to be \cite{guillin2021uniform}, which derives the uniform-in-time estimate $H(P^{n,n}_t\,|\,\mu^{\otimes n}_t)=O(1)$ for certain singular models on the torus by using similarly a LSI for $\mu_t$ which is uniform in $t$; by subadditivity, they deduce $H(P^{n,k}_t\,|\,\mu^{\otimes k}_t)=O(k/n)$ in \cite[Corollary 1]{guillin2021uniform}.

Let us briefly highlight the main new ideas of the method. We employ a well known relative entropy calculation, formally stated as follows. Suppose $(\mu^i_t)_{t \ge 0}$ is a probability measure flow solving the Fokker-Planck equation
\begin{align*}
\partial_t\mu^i_t = -\mathrm{div}(b^i_t\mu^i_t) + (\sigma^2/2)\Delta\mu^i_t,
\end{align*}
for some (time-dependent) vector field $b^i_t$, for $i=1,2$. Then
\begin{align}
\frac{d}{dt}H(\mu^1_t\,|\,\mu^2_t) &= \int \left((b^1_t - b^2_t) \cdot \nabla \log(d\mu^1_t/d\mu^2_t) - (\sigma^2/2)\big| \nabla \log(d\mu^1_t/d\mu^2_t)\big|^2 \right)\,d\mu^1_t. \label{intro:entropy}
\end{align}
See Lemma \ref{le:entropy-estimate} for a rigorous version.
For many purposes it is good enough to immediately bound the right-hand side of \eqref{intro:entropy} by $(1/2\sigma^2)\|b^1_t-b^2_t\|_{L^2(\mu^1_t)}^2$, which is actually precisely the time-derivative of the path-space relative entropy used in \cite{LackerFoundation}.
Here,  we avoid completing the square, instead using the LSI in the natural way to take advantage of the final term in \eqref{intro:entropy}.
More specifically, we apply \eqref{intro:entropy} with $b^1$ and $b^2$ respectively being the drifts of $P^{n,k}_t$ and $\mu^{\otimes k}_t$, the former identified using the BBGKY hierarchy. The $b^1-b^2$ term is estimated similarly to \cite{LackerFoundation}, and we ultimately obtain a differential inequality, central to our approach, of the form
\begin{align}
\frac{d}{dt}H_t^k \le - c_1 H_t^k + c_2\frac{k^3}{n^2} + c_3 k (H_t^{k+1} - H_t^k), \quad \text{where} \quad H^k_t := H(P^{n,k}_t\,|\,\mu^{\otimes k}_t), \label{intro:entropy-inequality}
\end{align}
for certain positive constants $c_1,c_2,c_3$ which do not depend on $(n,k)$. The key difference with \cite{LackerFoundation} is the first term, which stems from the LSI and provides the additional decay needed to obtain uniform-in-time bounds. Once \eqref{intro:entropy-inequality} is established and Gronwall's inequality is applied, the remainder of the proof proceeds by iterating the resulting integral inequality from $k$ to $n$, using the crude global estimate $\sup_t H^n_t =O(n)$ for the last step. 
Compared to the finite-time setting of \cite{LackerFoundation}, we face new difficulties in estimating the multiple integrals arising from this iteration in a sharp enough manner to produce the optimal exponent.

Two recent papers developed related quantitative estimates along the BBGKY hierarchy, using (weighted) $L^p$ norms rather than relative entropy, for non-exchangeable models in \cite{jabin2021mean} and for certain singular interactions in \cite{bresch2022new}. Lastly, the very recent \cite{Han2022} adapts the methods of \cite{LackerFoundation} to handle certain singular interactions. These results are all on finite time horizons.

In Section \ref{se:main-results} below we state precisely all of our main results. Sections \ref{se:sec-proof-main-theorem} and \ref{se:sec-proof-reversed-entropy} prove the two mains Theorems \ref{th:main-theorem} and \ref{th:reverse}, respectively, and Section \ref{se:sec-proof-corollaries} proves the corollaries.

\section{Main Result and Examples}\label{se:main-results}

The space of Borel probability measures on a metric space $E$ is denoted $\P(E)$. We use the notation $\langle \mu,f\rangle =\int_E f\,d\mu$ for integration.
For $\mu \in \P(E)$ and $k \in \N$, we write $\mu^{\otimes k}(dx_1,\ldots,dx_k)=\mu(dx_1)\cdots\mu(dx_k)$ for the product measure. We  use the notation $x=(x_1, \ldots, x_k)$ for general element $x\in E^k$. For example, if $\varphi:E^k\mapsto \R$, $\varphi(x)$ and $\varphi(x_1, \ldots, x_k)$ denote the same quantity. Similarly, $\mu(dx)$ and $\mu(dx_1,\ldots,dx_k)$ are equivalent notations for a measure $\mu\in\P(E^k)$.
For $\mu, \nu \in \P(\R^k)$, we define the relative entropy and the Fisher Information between $\mu$ and $\nu$ respectively by
\begin{align*}
H(\nu\,|\,\mu) := \int_{\R^k} \frac{d\nu}{d\mu} \log \frac{d\nu}{d\mu} d\mu, \qquad I(\nu\,|\,\mu) := \int_{\R^k} \left|\nabla \log \frac{d\nu}{d\mu} \right|^2 d\nu.
\end{align*}
We set $H(\nu\,|\,\mu):=\infty$ when $\nu\not\ll\mu$, and similarly $I(\nu\,|\,\mu):=\infty$ when $\nu\not\ll\mu$ or $\nabla \log d\nu/d\mu$ does not exist in $L^2(\nu)$.
Many measures on Euclidean space encountered in this paper are absolutely continuous with respect to Lebesgue measure. For such measures, we abuse notation by using the same letter to denote both a measure and its density, e.g., $\mu(dx)=\mu(x)dx$.
Finally, we define $C_c^\infty(\R^k)$ to be the space of infinitely differentiable functions with compact support on $\R^k$.

\subsection{General setup and main result}

Fix $d,n\in \N$. The $n$-particle system of interest \eqref{eq:intro-interacting-particle-system} is described by a weakly continuous flow of probability measures $(P^n_t)_{t \ge 0}$ on $(\R^d)^n$ satisfying the following Fokker-Planck equation, written in weak form:
For every $\varphi \in C^\infty_c((\R^d)^n)$ and $t \ge 0$,
\begin{align}
\langle & P^n_t - P^n_0,\varphi\rangle \label{eq:FokkerPlanck} \\
	& \ = \int_0^t \int_{(\R^d)^n}  \bigg[\sum_{i=1}^n\bigg(b_0(s, x_i) + \frac{1}{n-1} \sum_{j = 1, j\neq i}^n b(s, x_i, x_j)\bigg) \cdot \nabla_{x_i}\varphi(x) + \frac{\sigma^2}{2}\Delta \varphi(x)\bigg] P^n_s(dx)ds. \nonumber 
\end{align}
The mean field limit is described by a continuous flow of probability measures $(\mu_t)_{t \ge 0}$ on $\R^d$ satisfying the following nonlinear Fokker-Planck equation, again written in weak form:
For every $\varphi \in C^\infty_c(\R^d)$ and $t \ge 0$,
\begin{align}
\begin{split}
\langle \mu_t-\mu_0,\varphi\rangle &= \int_0^t \int_{\R^d}  \bigg[\bigg(b_0(s, x) + \langle \mu_s, b(s, x, \cdot)\rangle\bigg) \cdot \nabla \varphi(x) + \frac{\sigma^2}{2}\Delta \varphi(x)\bigg] \mu_s(dx)ds.
	\end{split} \label{eq:FokkerPlanck-nonlinear}
\end{align}
Recall that a function on a metric space is said to be \emph{locally bounded} if its restriction to any bounded set is bounded. A probability measure on $(\R^d)^n$ is \emph{exchangeable} if it is invariant under permutations of its $n$ coordinates.
Our first set of assumptions is technical in nature:

\begin{assumption}{\textbf{E}} \label{assumption:E}
{\ }
\begin{enumerate}
\item[(E.1)] We are given a scalar $\sigma > 0$, and Borel measurable functions $b_0 : [0,\infty) \times \R^d \to \R^d$ and $b : [0,\infty) \times \R^d \times \R^d \to \R^d$, where $b_0$ is locally bounded.
\item[(E.2)] There exists a weak solution $(\mu_t)_{t \ge 0}$ to the nonlinear Fokker-Planck equation \eqref{eq:FokkerPlanck-nonlinear} such that $b(t,x,\cdot) \in L^1(\mu_t)$ for all $(t,x)$, and $(t,x) \mapsto \langle \mu_t,b(t,x,\cdot)\rangle$ is locally bounded.
\item[(E.3)] There exists a weak solution $(P^n_t)_{t \ge 0}$ to the Fokker-Planck equation \eqref{eq:FokkerPlanck} such that $P^n_t$ is exchangeable for each $t > 0$. Moreover, for each $p,T > 0$,
\begin{align*}
\int_0^T\int_{(\R^d)^n}\big( | b(t,x_1,x_2)|^p + |\langle \mu_t,b(t,x_1,\cdot)|^2  \big)P^n_t(dx)dt &< \infty, \\
\sup_{t \in [0,T]} \int_{(\R^d)^n} \big( |b_0(t,x_1)|^2 + | b(t,x_1,x_2)|^2 \big) P^n_t(dx) &< \infty.
\end{align*}
\end{enumerate}
\end{assumption}

Note that we prefer to \emph{assume} the existence of $P^n_t$ and $\mu_t$, rather than placing assumptions on $(b_0,b)$ which imply existence. The assumptions on $(b_0,b)$ are thus mostly implicit, which makes our main result fairly general. Sections \ref{se:convex-study} and \ref{se:torus-study} give more concrete sufficient conditions, stated directly in terms of $(b_0,b)$. The assumptions of local boundedness in (E.1,2) and the $p$-integrability condition in (E.3) are purely technical, serving only to justify a relative entropy estimate (see Lemma \ref{le:entropy-estimate} below).
The next set of assumptions is the more essential one. 

\begin{assumption}{\textbf{A}} \label{assumption:A}
{\ }
\begin{enumerate}
\item[(A.1)] Log-Sobolev inequality (LSI): There exists a constant $\LSI >0$ such that 
\begin{align}
H(\nu \,|\, \mu_t) \leq \LSI  I(\nu \,|\, \mu_t),  \quad \forall \nu \in \P(\R^d), \  t \ge 0.
\label{eq:log-sobolev-inequality}
\end{align}
\item[(A.2)] Transport-type inequality: There exists $\gamma > 0$ such that
\begin{align}
\left| \langle \nu - \mu_t,  b(t, x, \cdot) \rangle \right|^2 \leq \gamma H(\nu \,|\, \mu_t), \quad \forall \nu \in \P(\R^d), \  x \in \R^d, \ t \ge 0.
\label{eq:transport_type_inequality}
\end{align}
\item[(A.3)] $L^2$-boundedness: We have
\begin{align}
M := \esssup_{t\geq 0}\int_{(\R^d)^n} \left|b(t, x_1, x_2) - \langle \mu_t, b(t, x_1, \cdot)\rangle \right|^2 P^n_t(dx) < \infty.
\label{eq:square_integrability}
\end{align}
\end{enumerate}
\end{assumption}

Define $P^k_t \in \P((\R^d)^k)$ to be the $k$-particle marginal of $P^n_t$. That is,  
\begin{align*}
\langle P^k_t,\varphi\rangle = \int_{(\R^d)^n} \varphi(x_1,\ldots,x_k) P^n_t(dx_1,\ldots,dx_n),
\end{align*}
for bounded measurable $\varphi : (\R^d)^k \to \R$.
For brevity, and because $n$ can be considered as fixed in the following non-asymptotic results, we write $P^k_t$ instead of $P^{n,k}_t$ as in the introduction.
The following is our first main result.

\begin{theorem}\label{th:main-theorem}
Suppose Assumptions \ref{assumption:E} and \ref{assumption:A} hold. Let $r_c := \frac{\sigma^4}{4\gamma \LSI }-1$.
\begin{enumerate}
\item Suppose that $r_c > 1$ and that there exists a constant $C_0>0$ such that 
\begin{align*}
H(P_0^k \,|\, \mu_0^{\otimes{k}}) \leq C_0 (k/n)^2, \quad \text{for all } k = 1, \ldots, n.
\end{align*}
Then there exists a constant $C>0$ depending only on $(\sigma, \gamma, \LSI ,M,C_0)$ such that
\begin{align*}
H(P_T^k \,|\, \mu_T^{\otimes{k}}) \leq C (k/n)^2, \quad \text{for all } T \ge 0, \ k=1, \ldots, n.
\end{align*}
\item Suppose that $0 < r_c \le 1$ and that for each $0<\epsilon_1 < \epsilon_2 < r_c$ there exists a constant $C_0^{\epsilon_1, \epsilon_2}>0$ such that 
\begin{align*}
H(P_0^k \,|\, \mu_0^{\otimes{k}}) \leq C_0^{\epsilon_1, \epsilon_2} k^{1+r_c - \epsilon_1} /n^{2r_c - \epsilon_2}, \quad \text{for all } k = 1, \ldots, n.
\end{align*}
Then, for any $0< \epsilon_1 <\epsilon_2 < r_c$, there exists a constant $C>0$ depending only on $(\sigma, \gamma, \LSI,M, \epsilon_1, \epsilon_2,C_0^{\epsilon_1, \epsilon_2}, C_0^{(\epsilon_2-\epsilon_1)/2, \epsilon_2 - \epsilon_1})$ such that:
\begin{align*}
H(P_T^k \,|\, \mu_T^{\otimes{k}}) \leq C k^{1 + r_c - \epsilon_1} /n^{2r_c - \epsilon_2}, \quad \text{for all }  T \ge 0, \ k=1, \ldots, n.
\end{align*}
\end{enumerate}
\end{theorem}

Note that Theorem \ref{th:main-theorem} only applies at sufficiently \emph{high temperature}, as it does not cover the case $\sigma^4 \le 4\gamma \LSI $ (equivalently, $r_c \le 0$). For high enough temperature, $\sigma^4 > 8\gamma \LSI $, Theorem \ref{th:main-theorem}(1) achieves the optimal order $(k/n)^2$.
See \cite[Example 2.8 and Section 3]{LackerFoundation} for a simple example (with linear coefficients) showing that the exponent $2$ cannot be improved. In the intermediate regime $4\gamma \LSI  < \sigma^4 \le 8 \gamma \LSI $, we do not know if our rate is sharp.

The assumption $H(P^k_0\,|\,\mu^{\otimes k})\le C_0(k/n)^2$ holds trivially in the case of i.i.d.\ initial conditions $P^n_0=\mu_0^{\otimes n}$. See the main results of \cite{LackerGibbs} for natural families of Gibbs measures for which this assumption holds non-trivially. The non-degenerate noise $\sigma> 0$ will ensure that the measures $P^n_t$ and $\mu_t$ are absolutely continuous for each $t > 0$ (see Remark \ref{re:mu2-positive}), but they need not be at $t=0$; in particular, Theorem \ref{th:main-theorem} can accommodate Dirac initial conditions.

\begin{remark} \label{re:kuramoto}
For small temperature, uniform-in-time propagation of chaos can fail, for the simple reason that the mean field equation may admit multiple invariant measures which all satisfy a LSI. In other words, the presence of Assumptions \ref{assumption:A} and \ref{assumption:E} alone are not enough to guarantee uniform-in-time propagation of chaos, without an additional smallness condition.
For a simple example, consider the Kuramoto model 
\begin{align}
dX^i_t = \frac{K}{ n-1}\sum_{j=1}^n \sin(X^i_t-X^j_t)dt + dB^i_t, \label{def:Kuramoto}
\end{align}
where $K > 0$ is a constant, and particles take values in the circle $\R/2\pi\Z \cong [0,2\pi]$.
This $n$-particle system is uniquely ergodic.
The uniform measure is always an invariant measure for the corresponding mean field model; i.e., $\mu_t(dx) = dx/2\pi$ for all $t$ solves \eqref{eq:FokkerPlanck-nonlinear}. In the supercritical case $K > 1$, the mean field limit admits an infinite set $S_{\mathrm{MF}}$ of invariant measures, obtained as the rotations of a common density which is bounded from above and below away from zero. See \cite{bertini2010dynamical} for details. In particular, all invariant measures admit a LSI by the Holley-Stroock perturbation argument \cite[Proposition 5.1.6]{Bakry}. One easily checks that Assumptions \ref{assumption:A} and \ref{assumption:E} hold when $\mu_t=\mu$ for all $t \ge 0$ for some $\mu \in S_{\mathrm{MF}}$, but uniform-in-time propagation of chaos cannot hold  in the sense of Theorem \ref{th:main-theorem} when $K > 1$. If it did, it would lead to the absurd conclusion that the $1$-particle marginal of the unique invariant measure of the $n$-particle system converges to $\mu$, for each $\mu \in S_{\mathrm{MF}}$. However, we highlight the remarkable recent results of \cite[Section 4]{Delarue2021}, which show that uniform-in-time propagation of chaos still holds modulo rotations, if one initializes away from the uniform measure which is unstable when $K>1$.
\end{remark}

\begin{remark} \label{re:OttoVillani}
The entropy bounds of Theorem \ref{th:main-theorem} imply similar bounds in Wasserstein distance. 
To be precise, recall first the definition of the $p$-Wasserstein distance between two measures $\mu, \nu \in \P(\R^d)$, for $p \ge 1$:
\begin{align}
\W_p(\mu, \nu) = \inf_\pi \left( \int_{\R^d \times \R^d} |x - y|^p \pi(dx, dy)\right)^{1/p}, \label{def:Wasserstein}
\end{align}
where the infimum is over $\pi \in \P(\R^d \times \R^d)$ with marginals $\mu$ and $\nu$. 
By a famous theorem of Otto-Villani \cite{OttoVillani} (see also \cite[Theorem 8.12]{CLeonard2009}), the LSI \eqref{eq:log-sobolev-inequality} implies  the quadratic transport inequality
\begin{align}
\W_2^2(\nu,\mu_t) \le 4\LSI H(\nu\,|\,\mu_t), \quad \forall \nu \in \P(\R^d), \ t \ge 0. \label{eq:OttoVillani}
\end{align}
The quadratic transport inequality  tensorizes \cite[Proposition 1.9]{CLeonard2009}, in the sense that 
\begin{align*}
\W_2^2(\nu,\mu_t^{\otimes k}) \le 4\LSI H(\nu\,|\,\mu_t^{\otimes k}), \quad \forall k \in \N, \ \nu \in \P((\R^d)^k), \ t \ge 0.
\end{align*}
In particular, in case (1) of Theorem \ref{th:main-theorem}, $\sup_{t \ge 0}\W_2(P^k_t,\mu^{\otimes k}_t) =O(k/n)$.
\end{remark}

\begin{remark}
We have not optimized or reported a precise value of the constant $C$ in Theorem \ref{th:main-theorem}, which is complicated and not very informative in general. However, in the case $\sigma^4>12\gamma \LSI $, it is not difficult to track the constants in our proof to obtain 
\begin{align*}
& H(P_T^k \,|\, \mu_T^{\otimes k}) \leq \big(C_1 + C_2 e^{- \sigma^2 T/24 \LSI}\big) (k/n)^2, \quad \forall T, k, \\
\text{where } \ C_1 &= \frac{10000M\sigma^{4} \gamma^{2} \LSI  }{(1-12 \gamma \LSI  \sigma^{-4})^2}, \qquad C_2 = 1250 \left(C_0 + \frac{\sqrt{\gamma M C_0} \LSI  \sigma^{-4}}{1 - 12 \gamma \LSI \sigma^{-4}}\right) \frac{\sigma^8}{\gamma^2 \LSI ^2}.
\end{align*}
This reveals that the term containing the constant $C_0$, which controls the initial entropy, decays as $T \to \infty$. Thus, as one would expect, the effect of the initial condition disappears over long time horizons.
By being even more careful in the proof, we can obtain $n$-dependent constants $C_1(n)$ and $C_2(n)$ in place of $C_1$ and $C_2$, which are bounded in $n$ and have the advantage of vanishing as either $\LSI  \to 0$, $M \to 0$, or $\sigma \to \infty$, for any fixed $n$ (but not uniformly in $n$).
Similar bounds are possible but much more involved in the remaining case $4 \gamma \LSI < \sigma^4\le 12\gamma \LSI $.
\end{remark}

\subsection{Reversing the relative entropy estimate}

This section presents a similar result for $H(\mu_t^{\otimes{k}} \,|\, P_t^k)$. We need a similar set of assumptions as before, essentially with the roles of the measures $P^n$ and $\mu$ inverted. We also add a simplifying assumption that the interaction function $b$ is bounded, which is a major limitation; see Remark \ref{re:rev-bounded} for discussion.

\begin{assumption}{\textbf{R}} \label{assumption:R}
{\ }
\begin{enumerate}
\item[(R.1)] We are given a scalar $\sigma > 0$, and Borel measurable functions $b_0 : [0,\infty) \times \R^d \to \R^d$ and $b : [0,\infty) \times \R^d \times \R^d \to \R^d$, where $b$ is bounded and $b_0$ is locally bounded.
\item[(R.2)] There exists a weak solution $(P^n_t)_{t \ge 0}$ to the Fokker-Planck equation \eqref{eq:FokkerPlanck} such that $P^n_t$ is exchangeable for each $t > 0$.
\item[(R.3)] There exists a weak solution $(\mu_t)_{t \ge 0}$ to the nonlinear Fokker-Planck equation \eqref{eq:FokkerPlanck-nonlinear} with 
\begin{align}
\sup_{t\in[0, T]} \int_{\R^d} |b_0(t,x)|^2 \mu_t(dx) < \infty, \quad T>0.
\label{eq:integ-b_0-rev}
\end{align}
\item[(R.4)] Log-Sobolev inequality (LSI): There exists a constant $\LSI >0$ such that 
\begin{align}
H(\nu \,|\, P_t^n) \leq \LSI  I(\nu \,|\, P_t^n), \quad \forall \nu \in \P((\R^d)^n), \ t \geq 0, \ n\in \N.
\label{eq:LSI-rev}
\end{align}
\end{enumerate}
\end{assumption}

\begin{theorem} \label{th:reverse}
Suppose that Assumption \ref{assumption:R} holds. Define
\begin{align*}
p_c = \frac{\sigma^4}{8 \LSI \||b|^2\|_\infty}.
\end{align*}
\begin{enumerate}
\item Suppose that $p_c > 2$ and that there exists a constant $C_0>0$ such that
\begin{align*}
H(\mu_0^{\otimes{k}} \,|\, P_0^k) \leq C_0 (k/n)^2, \quad \text{for all } k = 1, \ldots, n.
\end{align*}
Then there exists a constant $C>0$ depending only on $(\sigma, \||b|^2\|_\infty, \LSI ,C_0)$ such that
\begin{align*}
H(\mu_T^{\otimes{k}} \,|\, P_T^k) \leq C (k/n)^2, \quad \text{for all } T \ge 0, \ k=1, \ldots, n.
\end{align*}
\item Suppose that $p_c \leq 2$ and that there exist constants $C_0>0$ and $\epsilon \in (0, p_c)$ such that
\begin{align*}
H(\mu_0^{\otimes{k}} \,|\, P_0^k) \leq C_0 (k/n)^{p_c-\epsilon}, \quad \text{for all } k = 1, \ldots, n.
\end{align*}
Then there exists a constant $C>0$ depending only on $(\sigma,\||b|^2\|_\infty,\LSI ,C_0,\epsilon)$ such that
\begin{align*}
H(\mu_T^{\otimes{k}} \,|\, P_T^k) \leq C (k/n)^{p_c-\epsilon}, \quad \text{for all } T \ge 0, \ k=1, \ldots, n.
\end{align*}
\end{enumerate}
\end{theorem}

One can compare Theorem \ref{th:reverse} with Theorem \ref{th:main-theorem} by identifying $\gamma = 2\||b|^2\|_\infty$, noting that \eqref{eq:transport_type_inequality} holds with this constant by Pinsker's inequality.
The advantage of Theorem \ref{th:reverse} is that it applies no matter the value of $p_c = \sigma^4/4\gamma \LSI = 1+r_c$, whereas Theorem \ref{th:main-theorem} applies only when $p_c > 1$. Moreover, in the range $1 < p_c \le 2$ the bound $O((k/n)^{p_c-\epsilon})$ obtained in Theorem \ref{th:reverse} is better than the bound $O(k^{p_c-\epsilon_1}/n^{2(p_c-1)-\epsilon_2})$ of Theorem \ref{th:main-theorem}. When $p_c > 2$, the bound of $O((k/n)^2)$ obtained in each theorem is the same.
We do not know if the exponent $p_c-\epsilon$ is sharp in Theorem \ref{th:reverse}(2).
Let us note, similarly to Remark \ref{re:OttoVillani}, that the results of Theorem \ref{th:reverse} imply the bound on Wasserstein distance, by marginalizing the LSI \eqref{eq:LSI-rev}  and using the Otto-Villani theorem to  get $\W_2^2(\mu_t^{\otimes k},P^k_t) \le 4\LSI H(\mu_t^{\otimes k} \,|\, P^k_t)$ for $t \ge 0$.

\begin{remark} \label{re:rev-bounded}
The assumption in Theorem \ref{th:reverse} that $b$ is bounded is difficult to relax. It could likely be generalized to the following analogue of Assumption (\ref{assumption:A}.2):
\begin{align*}
\left| \langle \mu_t - P^{k+1|k}_{t,x},  b(t, x_1, \cdot) \rangle \right|^2 \leq \gamma H\big(\mu_t  \,|\, P^{k+1|k}_{t,x}\big), \  \  \forall 1 \le k < n, \ x=(x_1,\ldots,x_k) \in (\R^d)^k, \ t \ge 0,
\end{align*}
where $P^{k+1|k}_{t,x}(dx_{k+1})$ denotes the conditional law of $X^{k+1}_t$ given $(X^1_t,\ldots,X^k_t)=x$ under $P^n_t$.
These conditional measures do not seem tractable enough to enable a proof of functional inequalities of this form beyond the case of bounded $b$; cf.\ \cite[Remark 4.11]{LackerFoundation}.
\end{remark}

These main theorems give  recipes for quantitative uniform-in-time propagation of chaos, and the rest of Section \ref{se:main-results} describes concrete situations in which they apply. 
There are conceivably many situations in which these conditions can be checked, on a case-by-case basis. We give two somewhat general classes of examples below. The first deals with convex potentials, for which the well known Bakry-Emery framework yields log-Sobolev inequalities along dynamics. The second deals with a class of models set on the torus, where the LSI can be obtained using the Holley-Stroock perturbation lemma after showing that the density of $\mu_t$ is bounded from above and below away from zero uniformly in time.

\subsection{Convex potentials}\label{se:convex-study}

Our first Corollary \ref{cor:convex-main} provides a sharper rate of convergence than was previously known for the extremely well-studied case of convex potentials. We impose similar assumptions to \cite{Malrieu2008}, albeit with more restrictions on the interaction potential $W$.   We write $\succeq$ to denote positive definite (Loewner) order.

\begin{assumption}{\textbf{C}} \label{assumption:C}
Assume $b_0(t,x) = - \nabla U(x)$ and $b(t, x, y) = - \nabla W(x-y)$, where $U$ and $W$ are twice continuously differentiable functions satisfying the following:
\begin{enumerate}
\item[(C.1)] We have $\nabla^2(U+W) \succeq \alpha I$ for some $\alpha >0$, and each function $\psi=U$ and $\psi=W$ is convex at infinity in the sense that there exist constants $c_1^\psi,c_2^\psi \ge 0$ for which
\begin{align*}
(x-y) \cdot (\nabla \psi(x) - \nabla \psi(y)) \ge c_1^\psi|x-y|^2 - c_2^\psi, \qquad \forall x,y \in \R^d.
\end{align*}
We require \emph{uniform} convexity at infinity for $U$, in the sense that $c_1^U > 0$.
\item[(C.2)] There exist $C_U, p_U  > 0$ such that  $|\nabla U(x) | \le C_U(1 + |x|^{ p_U} )$ for all $x \in \R^d$.
\item[(C.3)] $W$ is even and $\min(L,R) < \infty$, where we define $R : =\||\nabla W|\|_\infty$ and $L= \sup_{x \neq y}|\nabla W(x)-\nabla W(y)|/|x-y|$.
That is, $\nabla W$ is either bounded or Lipschitz (or both).
\end{enumerate}
\end{assumption}

\begin{corollary}\label{cor:convex-main}
Suppose Assumption \ref{assumption:C} holds.
Let $\mu_0 \in \P(\R^d)$ satisfy the LSI 
\begin{align*}
H(\nu \,|\, \mu_0)\leq (\LSI _0/4) I(\nu \,|\, \mu_0), \quad \forall \nu \in \P(\R^d),
\end{align*}
for some $\LSI_0 > 0$.
Let $P^n_0 \in \P((\R^d)^n)$ be exchangeable and have finite moments of every order.
Then there exists a unique solution $(P^n_t)_{t \ge 0}$ of the Fokker-Planck equation \eqref{eq:FokkerPlanck} starting from $P^n_0$, and there exists a solution $(\mu_t)_{t \ge 0}$ of the nonlinear Fokker-Planck equation \eqref{eq:FokkerPlanck-nonlinear} starting from $\mu_0$, unique among the class of solutions satisfying $\int_0^T \int_{\R^d} |x|^p\mu_t(dx)dt < \infty$ for all $T,p > 0$.
Moreover, the following hold:
\begin{enumerate}
\item Assumptions \ref{assumption:A} and \ref{assumption:E} are satisfied, with $\LSI =\max(\LSI _0/4,\sigma^2/4\alpha)$, $\gamma=\min(4\LSI L^2,2R^2)$, and $M$ bounded by a finite constant depending only on $L$, $R$, $C_U$, and $\int_{(\R^d)^n}|x_1|^2P^n_0(dx)$. In particular, Theorem \ref{th:main-theorem} applies with
\begin{align*}
r_c &= \frac{\sigma^4}{4\gamma \LSI } - 1 = \max\Big(\min\Big(\frac{\sigma^4}{ \LSI _0^2 L^2},\, \frac{\alpha^2}{ L^2}\Big), \, \min\Big(\frac{\sigma^4}{2\LSI _0 R^2},\, \frac{\sigma^2 \alpha}{ 2 R^2}\Big) \Big) - 1.
\end{align*}
\item If $R < \infty$ and (C.1) holds with $c^U_2=c^W_2=c^W_1=0$ and $c^U_1=\alpha$ (i.e., $U$ is $\alpha$-uniformly convex and $W$ is convex), then Assumption \ref{assumption:R} is satisfied with $\LSI = \max(\eta_0/4,\sigma^2/4\alpha)$, and Theorem \ref{th:reverse} applies with
\begin{align*}
p_c = \min\Big(\frac{\sigma^4}{2\eta_0 R^2},\frac{\sigma^2\alpha}{2 R^2}\Big).
\end{align*}
\end{enumerate}
\end{corollary}

Note that $\LSI_0=0$ when $\mu_0$ is a Dirac.
In case (1), if $\LSI _0 \le \sigma^2/\alpha$ and $R=\infty$, then $r_c =(\alpha/L)^2-1$ and the two ranges of $r_c$ in Theorem \ref{th:main-theorem} correspond to $\alpha > \sqrt{2}L$ and $L < \alpha \le \sqrt{2} L$, respectively. We are unable to treat the case $\alpha \le L$, instead requiring the convexity to be sufficiently stronger than the interaction strength in the sense that $\alpha > L$. This kind of assumption is common in the literature, more when dealing with non-convex cases \cite{arnaudon2020second,Delarue2021,guillin2019uniform,LackerGibbs,salem2018gradient}. On the other hand, in the case where $\nabla W$ is bounded, part (2) has no smallness constraint anymore, except that the optimal exponent of $2$ is obtained only when $p_c \ge 2$, or $\sigma^2 \alpha \ge 4R^2$.

There are two main limitations in Corollary \ref{cor:convex-main}. The first is that part (1) requires $\alpha > L$, as discussed just above. The second is that it requires $\nabla W$ to be globally Lipschitz or bounded, which in particular rules out relevant cases like $W(x) = |x|^3$. Neither of these assumptions is needed in order to obtain \emph{qualitative} uniform-in-time propagation of chaos, which has been known since \cite{Malrieu2008,Malrieu2001}. 

\begin{remark}
Theorem \ref{th:main-theorem} and Corollary \ref{cor:convex-main} cover the Gaussian case, where $b_0(t,x) = -ax$ and $b(t, x, y) = b(y-x)$, for $x,y\in\R$, with $a,b>0$. Suppose that $\sigma = 1$ and that $\mu_0 = \delta_0$. 
It was shown in  \cite[Example 2.8 and Section 3]{LackerFoundation} that uniform-in-time propagation of chaos holds in this case at a rate of $O((k/n)^2)$, no matter the values of $a,b > 0$.
Corollary \ref{cor:convex-main} recovers this rate, but only when $a/b$ is large enough.
Indeed, this setting fits into Corollary \ref{cor:convex-main}(1), with $\eta_0=0$, $\alpha=a+b$, $L=b$ (and $R=\infty$), which yields the exponent $r_c=(1+(a/b))^2-1$.
We recover the optimal $O((k/n)^2)$ only when $r_c > 1$, or $a/b > \sqrt{2}-1$. Corollary \ref{cor:convex-main}(1) still applies but yields a suboptimal exponent when $a/b \le \sqrt{2}-1$.
\end{remark}

We note also that \cite{Durmus2020} recently used coupling methods to obtain quantitative propagation of chaos even when $U+W$ is non-convex, but with a suboptimal rate compared to Corollary \ref{cor:convex-main}. Their assumptions are somewhat similar otherwise, assuming a smallness condition on the Lipschitz constant of $\nabla W$, though they do not cover the case of non-Lipschitz $\nabla W$.

\subsection{Models on the torus}\label{se:torus-study}

In this section, we present a class of models in which the state space $\R^d$ is replaced by the torus $\T^d= \R^d/\Z^d$, and we take $b_0\equiv 0$ and $b(t,x,y)=K(x-y)$ for some vector field $K: \R^d \to \R^d$.
The proofs of Theorems \ref{th:main-theorem} and \ref{th:reverse} adapt without change to the case of the torus.
The McKean-Vlasov SDE and its corresponding PDE take the form
\begin{align}
dX_t &= K * \mu_t(X_t) dt + \sigma dB_t, \label{def:SDEtorus} \\
\partial_t\mu  &= -\div(\mu K * \mu  ) + \frac{\sigma^2}{2}\Delta \mu . \label{def:MVPDEtorus}
\end{align}
The $n$-particle Fokker-Planck equation becomes
\begin{align}
\partial_t P^n_t(x) = -\sum_{i=1}^n \div_{x_i}\bigg(\frac{1}{n-1}\sum_{j \neq i} K(x_i-x_j) P^n_t(x) \bigg) + \frac{\sigma^2}{2}\Delta P^n_t(x). \label{def:PDEtorus}
\end{align}
We make the following assumptions:

\begin{assumption}{\textbf{T}} \label{assumption:T}
Assume that $K$ is Lipschitz and also that the  initial law $\mu_0$ admits a smooth density satisfying the pointwise bound $\lambda^{-1} \le \mu_0 \le \lambda$, for some $\lambda \ge 1$. 
\end{assumption}

Note since that $\mu \equiv 1$ solves the PDE \eqref{def:MVPDEtorus}, because $\mu K * \mu \equiv \int_{\T^d} K$ is a constant; that is, the uniform measure on $\T^d$ is invariant for \eqref{def:MVPDEtorus}. There may be additional invariant measures, in general.
Our main result in this section gives a uniform-in-time propagation of chaos, with a sharp rate, for a sufficiently small mean field interaction $K$, which in particular rules out the existence of additional invariant measures:

\begin{corollary} \label{co:torus}
Suppose Assumption \ref{assumption:T} holds, and let $P^n_0 \in \P((\T^d)^n)$ be arbitrary. Then there exists a unique weak solution $P^n$ of \eqref{def:PDEtorus} starting from $P^n_0$, and there exists a unique classical solution $\mu$ of \eqref{def:MVPDEtorus} starting from $\mu_0$.
Assume that $\div \, K$ is small  in the sense that 
\begin{align}
\|\div \, K\|_\infty < \frac{\sigma^2\pi^2}{1+2\sqrt{2\log \lambda}} < \sigma^2\pi^2. \label{asmp:Ksmall}
\end{align}
Then Assumptions \ref{assumption:A} and \ref{assumption:E} hold, and Theorem \ref{th:main-theorem} applies with
\begin{align*}
r_c = \frac{\sigma^4(1-2r_0)}{2\lambda^2  \mathrm{diam}^2(K) } - 1,
\end{align*}
where we define
\begin{align*}
r_0 := \frac{\|\div\, K\|_\infty\sqrt{ 2 \log \lambda}}{\sigma^2\pi^2 - \|\div\, K\|_\infty}, \quad \mathrm{diam}(K) := \sup_{x,z \in \T^d}|K(x)-K(z)|.
\end{align*}
\end{corollary}

Note that \eqref{asmp:Ksmall} ensures that $r_0 < 1/2$, so that $r_c > -1$. Theorem \ref{th:main-theorem} only yields anything useful, of course, if $r_c > 0$, which can occur when some combination of the following effects are present: $\sigma$ is large, or $K$ has small oscillations or divergence, or $\lambda$ is close to 1  which means that the initial law $\mu_0$ is $L^\infty$-close to the uniform distribution. Note when $K$ is divergence-free the simplification $r_0=0$ occurs.

\begin{example} \label{ex:Kuramoto}
Corollary \ref{co:torus} notably covers the example of the Kuramoto model  (without disorder) discussed in Remark \ref{re:kuramoto}, in sufficiently subcritical regimes. We keep this discussion brief, referring to \cite{acebron2005kuramoto,bertini2010dynamical} for background on this model. The Kuramoto model \eqref{def:Kuramoto} was written for $X^i_t$ taking values in $[0,2\pi]$, and to fit it into Corollary \ref{co:torus} (set in $\T^1=\R/\Z\cong [0,1]$) we simply rescale $X^i_t$ to $X^i_t/2\pi$.
This yields the parameters $\sigma=1/2\pi$ and $K(x)=K\sin(2\pi x)/2\pi$.
In Corollary \ref{co:torus}, the condition \ref{asmp:Ksmall} becomes
\begin{align*}
 K < \frac{1}{4+8\sqrt{2\log \lambda}} < \frac14,
\end{align*}
and the constant $r_c$ becomes
\begin{align*}
r_c=r_c(K,\lambda) := \frac{1-4K(1+2\sqrt{2\log\lambda})}{32\pi^2\lambda^2K^2(1-4K)} - 1.
\end{align*}
We obtain quantitative uniform-in-time propagation of chaos when $K<K_c^0(\lambda)$, where we define $K_c^a(\lambda)$ for $a \ge 0$   as the ($\lambda$-dependent) value of $K$ for which $r_c(K,\lambda) = a$. For $K<K_c^1(\lambda)$ we obtain the optimal exponent $O((k/n)^2)$ from Theorem \ref{th:main-theorem}.
For example, when $\lambda=1$ (meaning the initialization $\mu_0$ is uniform), we have $K^0_c(1)=\sqrt{2}/8\pi$ and $K^1_c(1)=1/8\pi$. We are not able to determine the sharp rate quantitative uniform-in-time propagation of chaos all the way to criticality, i.e., for all $K < 1$.
\end{example}

\section{Proof of the main theorem}\label{se:sec-proof-main-theorem}

This section gives the proof of Theorem \ref{th:main-theorem}. We suppose that Assumptions \ref{assumption:E} and \ref{assumption:A} hold throughout this section. The central quantity of study will be the relative entropy
\begin{equation*}
H^k_t := H(P^k_t\,|\,\mu^{\otimes k}_t), \quad t \ge 0, \ 1 \le k \le n.
\end{equation*}

\subsection{The entropy estimate} \label{se:entropy_estimate}
The first step is to estimate the time derivative of $H^k_t$. We apply the following more general estimate between solutions of Fokker-Planck equations.

\begin{lemma} \label{le:entropy-estimate}
Let $d \in \N$ and $\sigma > 0$. For each $i=1,2$, let $b^i : [0,\infty) \times \R^d \to \R^d$ be measurable, and assume we are given a continuous flow of probability measures $(\mu^i_t)_{t \ge 0}$ on $\R^d$ satisfying the weak Fokker-Planck equation
\begin{align*}
\begin{split}
\langle \mu^i_t ,\varphi\rangle &= \langle  \mu^i_0,\varphi\rangle + \int_0^t \int_{\R^d} \Big( b^i(s,x)\cdot \nabla \varphi(x) + \frac{\sigma^2}{2}\Delta\varphi(x)\Big)\mu^i_s(dx)ds, \quad \forall \varphi \in C^\infty_c(\R^d), \ t \ge 0,
\end{split}
\end{align*}
as well as the following conditions:
\begin{enumerate}
\item[(H.1)] The function $b^2$ is locally bounded. 
\item[(H.2)] The function $b^1$ belongs to $L^p_{\mathrm{loc}}( \mu^1 )$ for some $p > d+2$. That is, for each $T > 0$ and each bounded Borel set $S \subset \R^d$ we have 
\begin{align*}
\int_0^T  \int_S & |b^1(t,x)|^p \mu_t^1(dx) dt < \infty.
\end{align*}
\item[(H.3)] The following hold for each $T > 0$:
\begin{align*}
\int_0^T  \int_{\R^d} & |b^2(t,x)|^2 \mu_t^1(dx) dt < \infty,  \qquad \sup_{t \in [0,T]}  \int_{\R^d} |b^1(t,x)|^2 \mu_t^1(dx)  < \infty. 
\end{align*}
\end{enumerate}
Then, it holds for each $t > s \ge 0$ that
\begin{align*}
H(\mu_t^1 \,|\, \mu_t^2) +  \frac{\sigma^2}{4} \int_s^t I(\mu_u^1 \,|\, \mu_u^2) \, du \leq H(\mu_s^1 \,|\, \mu_s^2) + \frac{1}{\sigma^2}\int_s^t \int_{\R^d} \left| b^1(u, x) - b^2(u, x)\right|^2 \mu^1_u(dx) du .
\end{align*}
In particular, if $H(\mu^1_0 \,|\, \mu^2_0) < \infty$, then $H(\mu_t^1 \,|\, \mu_t^2)<\infty$ for all $t > 0$.
\end{lemma}

Estimates of this form might be considered folklore, with similar forms appearing in \cite{bresch2020mean,jabin-wang-bounded,jabin-wang-W1inf} under different assumptions.
If one ignores questions of smoothness and integrability, the proof follows by applying the Fokker-Planck equation for $\mu^2$ with the test function $\varphi=(\mu^1/\mu^2)\log(\mu^1/\mu^2)$.
It takes some care, though, to make this rigorous.
Lemma \ref{le:entropy-estimate} follows almost immediately from the arguments in \cite[Lemma 2.4]{Bogachev2016}, with their main results holding under the sole assumptions that $b^1$ and $b^2$ are Borel measurable and locally bounded. Local boundedness of $b^1$ turns out to be too restrictive for our application, where we wish to use $b^1=\widehat{b}^k$ defined in \eqref{def:bhat^k_i} below; the issue is that local boundedness is not preserved under conditioning. Nonetheless, the assumptions of Lemma \ref{le:entropy-estimate} are enough to ensure sufficient smoothness of $\mu^1$, summarized in Remark \ref{re:mu2-positive} below, which enables the same arguments as in \cite{Bogachev2016} to be carried out, after an additional mollification step in which $b^1$ is approximated by a locally bounded function.
See Appendix A of the first arXiv version of the present paper for full details.


\begin{remark} \label{re:mu2-positive}
The assumptions (H.1,2) ensure that the measure $dt\mu^i_t(dx)$ admits a (H\"older) continuous density on $(0,\infty) \times \R^d$, which we denote by $\mu^i(t,x)$, for each $i=1,2$. In addition, $\mu^i(t,\cdot) \in W^{p,1}(U)$ for every $p \ge 1$, every $t > 0$, and every bounded open set $U \subset \R^d$. 
See \cite[Proposition 6.5.1]{bogachev-book}.
Lastly, since $b^2$ is locally bounded, the continuous version of the density $\mu^2$ is strictly positive on $(0,\infty) \times \R^d$; see \cite[Example 8.3.8]{bogachev-book}.
\end{remark}

\begin{remark}\label{re:enhanced-entropy-estimate}
An easy adaptation of the proof of Lemma \ref{le:entropy-estimate} gives the following more general estimate, for $0 < c < 1$:
\begin{align*}
H(\mu_t^1 | \mu_t^2) +  \frac{\sigma^2}{2}(1-c) \!\int_s^t I(\mu_u^1 | \mu_u^2) du \leq H(\mu_s^1 | \mu_s^2) + \frac{1}{2\sigma^2 c}\int_s^t \int_{\R^d} \left| b^1(u, x) - b^2(u, x)\right|^2 \mu^1_u(dx) du .
\end{align*}
This allows one to trade off between the Fisher information and drift terms. Remarkably, it turns out in our context that the optimal choice is always $c=1/2$.
\end{remark}

\subsection{The BBGKY hierarchy} \label{se:hierarchy}

We will apply Lemma \ref{le:entropy-estimate} with $P^k_t$ and $\mu^{\otimes k}_t$ in place of $\mu^1_t$ and $\mu^2_t$, respectively, but first we need to represent $P^k$ and $\mu^{\otimes k}$ for each $k$ as solutions of Fokker-Planck equations. For $\mu^{\otimes k}$ this is straightforward. For $P^k$  this is accomplished using the so-called \emph{BBGKY hierarchy}, which is well known, but we derive it in the lemma below for completeness; see \cite[Section 1.5]{golse2016dynamics} for additional references and a derivation in the zero-noise case. First, note that the measure $dtP^n_t(dx)$ has a positive density on $(0, \infty) \times (\R^d)^n$; see Remark \ref{re:mu2-positive} and note that (H.2) holds because of the local boundedness of $b_0$ and because of Assumption (\ref{assumption:E}.3). By marginalization, $dtP^k_t(dx)$ has a positive density on $(0, \infty) \times (\R^d)^k$. Recall in the following that we use the same letter to denote a measure and its density when it exists, e.g., $P^k_t(dx)=P^k_t(x)dx$.

\begin{lemma} \label{le:BBGKY}
Let $1 \le k < n$.
Define the conditional density
\begin{align*}
P^{k+1|k}_{t,x_1,\ldots,x_k}(x_{k+1}) = \frac{P^{k+1}_t(x_1,\ldots,x_{k+1})}{P^{k}_t(x_1,\ldots,x_{k})}, \quad t > 0, \ x_1,\ldots,x_{k+1} \in \R^d.
\end{align*}
For $i=1,\ldots,k$, define the functions $\widehat{b}^k_i : [0,\infty) \times (\R^d)^k \to \R^d$ by
\begin{align}
\widehat{b}^k_i(t,x_1,\ldots,x_k) := \frac{1}{n-1}\sum_{j \neq i,\,j=1}^k b(t,x_i,x_j) + \frac{n-k}{n-1}\langle P^{k+1|k}_{t,x_1,\ldots,x_k},b(t,x_i,\cdot)\rangle. \label{def:bhat^k_i}
\end{align}
Then, for each $1 \le k < n$, $\varphi \in C^\infty_c((\R^d)^k)$, and $t \ge 0$,
\begin{align*}
\begin{split}
\langle P^k_t - P^k_0,\varphi\rangle = \int_0^t \int_{(\R^d)^k}  \Big( \sum_{i=1}^k(b_0(s,x_i) + \widehat{b}^k_i(s,x))\cdot \nabla_{x_i} \varphi(x) + \frac{\sigma^2}{2}\Delta\varphi(x)\Big)P^k_s(dx)ds.
\end{split}
\end{align*}
\end{lemma}
\begin{proof}
Apply the weak Fokker-Planck equation \eqref{eq:FokkerPlanck} to a test function $\varphi \in C^\infty_c((\R^d)^k)$ depending on only the first $k$ variables to find
\begin{align*}
\langle P^k_t-P^k_0,\varphi\rangle  
	&= \sum_{i=1}^k\int_0^t \int_{(\R^d)^n} \! \bigg(b_0(s, x_i) + \frac{1}{n-1} \sum_{j = 1, j\neq i}^n b(s, x_i, x_j)\bigg) \!\cdot \!\nabla_{x_i}\varphi(x_1,\ldots,x_k) \,P^n_s(dx)ds \\
	&\quad + \frac{\sigma^2}{2}\int_0^t \int_{(\R^d)^n}  \Delta \varphi(x_1,\ldots,x_k) P^n_s(dx)ds.
\end{align*}
Marginalizing yields
\begin{align*}
\int_{(\R^d)^n} b_0(s, x_i)\cdot \nabla_{x_i}\varphi(x_1,\ldots,x_k)\,P^n_s(dx) &= \int_{(\R^d)^k} b_0(s, x_i)\cdot \nabla_{x_i}\varphi(x)\,P^k_s(dx), \\
\int_{(\R^d)^n} \Delta \varphi(x_1,\ldots,x_k)\,P^n_s(dx) &= \int_{(\R^d)^k} \Delta \varphi(x)\,P^k_s(dx).
\end{align*}
For the interaction term, we compute
\begin{align}
\int_{(\R^d)^n} & \frac{1}{n-1} \sum_{j = 1, j\neq i}^n b(s, x_i, x_j) \cdot \nabla_{x_i}\varphi(x_1,\ldots,x_k) \, P^n_s(dx) \nonumber \\
	&= \frac{1}{n-1} \sum_{j = 1, j\neq i}^k\int_{(\R^d)^n}  b(s, x_i, x_j) \cdot \nabla_{x_i}\varphi(x_1,\ldots,x_k) \, P^n_s(dx) \label{pf:BBGKY1} \\
	&\quad + \frac{1}{n-1} \sum_{j = k+1}^n\int_{(\R^d)^n}  b(s, x_i, x_j) \cdot \nabla_{x_i}\varphi(x_1,\ldots,x_k) \, P^n_s(dx) \nonumber 
\end{align}
We claim this is equal to $\int_{(\R^d)^k} \widehat{b}^k_i(t,x) \cdot \nabla_{x_i}\varphi(x) \,P^k_s(dx)$.
Again by marginalizing, the first term on the right-hand of \eqref{pf:BBGKY1} can be simplified to
\begin{align*}
\frac{1}{n-1} \sum_{j = 1, j\neq i}^k\int_{(\R^d)^k}  b(s, x_i, x_j) \cdot \nabla_{x_i}\varphi(x) \, P^k_s(dx),
\end{align*}
Using symmetry, the second term on the right-hand of \eqref{pf:BBGKY1} can be simplified to
\begin{align*}
&\frac{n-k}{n-1} \int_{(\R^d)^{k+1}}  b(s, x_i, x_{k+1}) \cdot \nabla_{x_i}\varphi(x_1,\ldots,x_k) \, P^{k+1}_s(dx) \\
	& \ \ \ = \frac{n-k}{n-1} \int_{(\R^d)^k} \bigg[ \int_{\R^d}  b(s, x_i, x_{k+1}) \,P^{k+1|k}_{s,x_1,\ldots,x_k}(dx_{k+1}) \bigg] \cdot \nabla_{x_i}\varphi(x_1,\ldots,x_k) \, P^k_s(dx_1,\ldots,dx_k).
\end{align*}
The expression in brackets is exactly $\langle P^{k+1|k}_{t,x_1,\ldots,x_k},b(t,x_i,\cdot)\rangle$, and the proof is complete.
\end{proof}

\begin{remark}
Alternatively, working at the level of the stochastic processes \eqref{eq:intro-interacting-particle-system}, the BBGKY hierarchy can be derived using the so-called \emph{mimicking theorem} \cite[Corollary 3.7]{brunick2013mimicking}.
\end{remark}

\subsection{Estimating $H^k_t$}

We now apply Lemma \ref{le:entropy-estimate} to estimate $H^k_t=H(P^k_t\,|\,\mu^{\otimes k}_t)$. We know from Lemma \ref{le:BBGKY} that $P^k$ solves a Fokker-Planck equation with drift $(\widehat{b}^k_1,\ldots,\widehat{b}^k_k)$. On the other hand, the nonlinear-Fokker-Planck equation \eqref{eq:FokkerPlanck-nonlinear} easily tensorizes to show that $\mu^{\otimes k}$ satisfies the following Fokker-Planck equation: for all $\varphi \in C^\infty_c((\R^d)^k)$ and $t \ge 0$,
\begin{align*}
\begin{split}
\langle \mu^{\otimes k}_t-\mu^{\otimes k}_0,\varphi\rangle  
	&= \int_0^t \int_{(\R^d)^k} \Big( \sum_{i=1}^k \big(b_0(s,x_i) + \langle \mu_s,b(s,x_i,\cdot)\rangle \big)\cdot \nabla_{x_i} \varphi(x) + \frac{\sigma^2}{2}\Delta\varphi(x)\Big)\mu^{\otimes k}_s(dx)ds.
\end{split}
\end{align*}
We now apply  Lemma \ref{le:entropy-estimate} to get, for all $t > s \ge 0$,
\begin{align}
H^k_t &- H^k_s + \!\frac{\sigma^2}{4}\int_s^tI(P^k_u\,|\,\mu^{\otimes k}_u)\,du  \le \!\frac{1}{\sigma^2}\!\int_s^t \!\int_{(\R^d)^k} \sum_{i=1}^k \big|\widehat{b}^k_i(u,x) - \langle \mu_u,b(u,x_i,\cdot)\rangle\big|^2 P^k_u(dx)du. \label{pf:main-ent-est1}
\end{align}
Indeed, the assumptions of Lemma \ref{le:entropy-estimate} are easily checked using Assumption \ref{assumption:E}: We know by (\ref{assumption:E}.1,2) that $b_0(t,x) + \langle \mu_t,b(t,x,\cdot)\rangle$ is locally bounded, so that (H.1) holds. The assumption (H.2) holds because $b_0$ is locally bounded, and because (\ref{assumption:E}.3) gives
\begin{align*}
\int_0^T \int_{(\R^d)^k} |\widehat{b}^k_i(t,x)|^p P^k_t(dx)dt \le \int_0^T \int_{(\R^d)^n}|b(t,x_1,x_2)|^p\,P^n_t(dx)dt < \infty,
\end{align*}
for all $p,T > 0$ and $i=1,\ldots,k$ by Jensen's inequality and by exchangeability. The first part of (H.3) comes from the local boundedness of $b_0$ and the first part of (\ref{assumption:E}.3). Similarly, the second part of (H.3) follows from the second part of (\ref{assumption:E}.3). 
With \eqref{pf:main-ent-est1} now justified, the integrability assumptions in \ref{assumption:E} will henceforth play no role.

The following lemma summarizes the implications of \eqref{pf:main-ent-est1}, which will be used in multiple ways. Recall the constants $(\LSI,\gamma,M)$ defined in Assumption \ref{assumption:A}.

\begin{lemma}\label{le:first_entropy_estimate}
Suppose that there exist constants $C_0 >0$ and $p_1,p_2\in(0, 2]$ such that
\begin{align*}
H^k_0 &\leq C_0 k^{p_1} / n^{p_2}, \quad 1 \le k \le n.
\end{align*}
Let $\delta > 0$, and define  $Z := \sigma^2/4\LSI $, $\tilde{\gamma} := \gamma (1+ \delta)/\sigma^2$, and  the quantities
\begin{align}
\tilde{A}_k^\ell(t_k) & := \bigg(  \prod_{j=k}^\ell \tilde{\gamma} j \bigg) \int_0^{t_k} \int_0^{t_{k+1}} \cdots \int_0^{t_\ell} e^{-\sum_{j=k}^\ell (Z + \tilde{\gamma} j)(t_{j} - t_{j+1})} dt_{\ell+1} \ldots dt_{k+1}\label{eq:def_overline_A_k_ell},\\
\tilde{B}_k^\ell(t_k) & := \bigg( \prod_{j=k}^{\ell-1} \tilde{\gamma} j \bigg) \int_0^{t_k} \int_0^{t_{k+1}} \cdots \int_0^{t_{\ell-1}} e^{- (Z + \tilde{\gamma} \ell)t_\ell -  \sum_{j=k}^{\ell-1} (Z + \tilde{\gamma} j)(t_{j} - t_{j+1})} dt_{\ell} \ldots dt_{k+1},\label{eq:def_tilde_B_k_ell}
\end{align}
for $t_k \ge 0$  and integers $\ell \ge k$, with the convention $\tilde{B}_k^k(t) := e^{-(Z + \tilde{\gamma}k)t}$.
Then the following hold:
\begin{enumerate}
\item For each $T \geq 0$ and $1 \le k < n$, 
\begin{align}
H_T^n &\leq C_0 n^{p_1 - p_2} e^{-\sigma^2T/4\eta} + 4 n M \LSI  \sigma^{-4}.
\label{eq:first_estimate_on_H_n} \\
H_T^k &\leq  \frac{C_0}{n^{p_2}}\sum_{\ell=k}^{n-1} \ell^{p_1} \tilde{B}_k^\ell(T) +\frac{ M}{\delta \gamma  n^2} \sum_{\ell=k}^{n-1} \ell^{2 } \tilde{A}_k^\ell(T)+ \tilde{A}_{k}^{n-1}(T) H_T^n.
\label{eq:first_estimate_on_H_k}
\end{align}
\item If we assume additionally that there exist $C_1 > 0$ and $p\in [0, 2]$ such that
\begin{align}
\sup_{T \ge 0} H_T^3 & \leq C_1 / n^p, \label{assmp:H3}
\end{align}
then, by setting $C_2 = 2 M + \sqrt{\gamma M C_1}$, it holds for each $T \ge 0$ and $1 \le k < n$  that
\begin{align}
H_T^n &\leq C_0 n^{p_1 - p_2} e^{-\sigma^2T/4\eta} + 4 n^{1-\frac{p}{2}} C_2 \LSI  \sigma^{-4}, \label{eq:sharpest_estimate_on_H_n} \\
H_T^k &\leq  \frac{C_0}{n^{p_2}}\sum_{\ell=k}^{n-1} \ell^{p_1} \tilde{B}_k^\ell(T) +\frac{ C_2}{\delta \gamma  n^2} \sum_{\ell=k}^{n-1} \ell^{2-\frac{p}{2}} \tilde{A}_k^\ell(T)+ \tilde{A}_{k}^{n-1}(T) H_T^n.
\label{eq:sharpest_estimate_on_H_k}
\end{align}
\end{enumerate}
\end{lemma}

Although we have no a priori information about $H^3_T$, the purpose of part (2) of Lemma \ref{le:first_entropy_estimate} is as follows. In the case of $r_c > 2$ in Theorem \ref{th:main-theorem}, we use Lemma \ref{le:first_entropy_estimate}(1) on a first pass through the argument, which ultimately leads to a bound of $H^k_T = O(k^3/n^2)$. This is suboptimal in the exponent of $k$. However, it implies $H^3_T = O(1/n^2)$, and we may then repeat the argument and apply the sharper Lemma \ref{le:first_entropy_estimate}(2) with $p=2$ to reduce the exponent on $k$ by one. This is carried out in Section \ref{se:pf-part(1)} and is similar to an argument in \cite{LackerFoundation}. The case where $0 < r_c < 2$ is trickier, because a first pass through the argument gives the worse bound $H^3_T = O(1/n^{r_c})$. The idea is to then apply Lemma \ref{le:first_entropy_estimate}(2) inductively. The exponent in $H^3_T=O(1/n^{r_c})$ improves after $m$ iterations from $r_c$ to $2\min(1,r_c(1-2^{-m}))$. If $r_c > 1$ then finitely many iterations yield the optimal exponent $2$, whereas if $r_c \le 1$ then the exponent merely approaches $2r_c$. (The constants grow with each iteration, so we may only iterate finitely many times.) This explains the exponent on $n$ in Theorem \ref{th:main-theorem}(2). Lemma \ref{le:iteration_Lemma} below summarizes the induction step, and Sections \ref{se:pf-part(1)-2} and \ref{se:pf-part(2)} implement it to complete the proof of Theorem \ref{th:main-theorem}.

\begin{proof}[Proof of Lemma \ref{le:first_entropy_estimate}]
We simplify the estimate \eqref{pf:main-ent-est1}, starting with the Fisher information term.
The LSI \eqref{eq:log-sobolev-inequality} is well known to tensorize \cite[Proposition 5.2.7]{Bakry}:
\begin{align*}
H(\nu \,|\, \mu_t^{\otimes{k}}) \leq \LSI  I(\nu \,|\, \mu_t^{\otimes{k}}), \qquad \forall t \geq 0, \  1 \le k \le n, \ \nu \in \P((\R^d)^k).
\end{align*}
Use this in inequality \eqref{pf:main-ent-est1} along with exchangeability to obtain
\begin{align}
H_t^k - H_s^k &\leq \frac{1}{\sigma^2} \int_s^t \int_{(\R^d)^k} \sum_{i=1}^k \left|\widehat{b}^k_i(u, x) - \langle \mu_u, b(u, x_i, \cdot) \rangle\right|^2 P^k_u(dx) du - \frac{\sigma^2}{4\LSI } \int_s^t H_u^k \, du \nonumber \\
	&= \frac{k}{\sigma^2} \int_s^t \int_{(\R^d)^k}  \left|\widehat{b}^k_1(u, x) - \langle \mu_u, b(u, x_1, \cdot) \rangle\right|^2 P^k_u(dx) du - \frac{\sigma^2}{4\LSI } \int_s^t H_u^k \, du.
 \label{eq:applying_estimate}
\end{align}
Let $t > 0$. Using the definition of $\widehat{b}^k_1$ from Lemma \ref{le:BBGKY}, we have
\begin{align*}
&\int_{(\R^d)^k}   \left|\widehat{b}^k_1(t, x) - \big\langle \mu_t, b(t, x_1, \cdot) \big\rangle\right|^2 P^k_t(dx) \\
& \ = \int_{(\R^d)^k}\bigg|\frac{1}{n-1}\sum_{j=2}^k \big( b(t,x_1,x_j) - \langle\mu_t, b(t, x_1, \cdot) \rangle\big) + \frac{n-k}{n-1}\big\langle P^{k+1|k}_{t,x} -  \mu_t, b(t, x_1, \cdot) \big\rangle\bigg|^2 P^k_t(dx).
\end{align*}
Note that the above is valid for any $1 \le k \le n$ with the convention that the second term inside the square equals $0$ for $k=n$. Use the inequality $(x+y)^2 \le (1+\delta)x^2 + (1+\delta^{-1})y^2$, valid for any $\delta>0$ and $x,y \in \R$, to bound this further by $(1+\delta^{-1})I+(1+\delta )II$, where we set
\begin{align*}
I &:= \int_{(\R^d)^k}\bigg|\frac{1}{n-1}\sum_{j=2}^k \big( b(t,x_1,x_j) - \langle\mu_t, b(t, x_1, \cdot) \rangle\big) \bigg|^2 P^k_t(dx), \\
II &:= \int_{(\R^d)^k}\left| \frac{n-k}{n-1}\big\langle P^{k+1|k}_{t,x} -  \mu_t, b(t, x_1, \cdot) \big\rangle\right|^2 P^k_t(dx).
\end{align*}

Let us first investigate the second term, assuming $k<n$ since it vanishes otherwise. Discard the constant $(n-k)/(n-1) \le 1$, and use Assumption (\ref{assumption:A}.2) to obtain
\begin{equation}
II \le \gamma \int_{(\R^d)^k}H\big(P^{k+1|k}_{t,x} \,\big|\, \mu_t \big)\, P^k_t(dx) = \gamma \big(H^{k+1}_t - H^k_t\big), \label{eq:first_bound_drift_second_term}
\end{equation}
with the second identity being the so-called \emph{chain rule} for relative entropy; indeed, this comes from the formula $P^{k+1}_t(x,x_{k+1})=P^k_t(x)P^{k+1|k}_{t,x}(x_{k+1})$ and the simple calculation
\begin{align*}
\int_{(\R^d)^k}H\big(P^{k+1|k}_{t,x} &\,\big|\, \mu_t \big)\, P^k_t(x)dx = \int_{(\R^d)^k} \bigg[\int_{\R^d} \log \frac{P^{k+1|k}_{t,x}(x_{k+1})}{\mu_t(x_{k+1})}\,P^{k+1|k}_{t,x}(x_{k+1})\,dx_{k+1}\bigg] P^k_t(x)dx \\
	&= \int_{(\R^d)^{k}}\int_{\R^d} \bigg[\log \frac{P^{k+1}(x,x_{k+1})}{\mu_t^{\otimes(k+1)}(x,x_{k+1})} - \log \frac{P^{k}(x)}{\mu_t^{\otimes k}(x)} \bigg] P^{k+1}_t(x,x_{k+1})dx_{k+1}dx \\
	&= H^{k+1}_t-H^k_t,
\end{align*}
where we write $(x,x_{k+1})$ here for a typical element of $(\R^d)^{k+1} \cong (\R^d)^k \times \R^d$.

We next study the first term, $I$.
For case (1) of the lemma, we use convexity of $|\cdot|^2$ to get
\begin{equation}
I \le \frac{ k-1 }{(n-1)^2}\int_{(\R^d)^k} \sum_{j=2}^k \big| \big( b(t,x_1,x_j) - \langle\mu_t, b(t, x_1, \cdot) \rangle\big) \big|^2 P^k_t(dx) \leq \frac{(k-1)^2}{(n-1)^2}M \le \frac{k^2}{n^2}M, \label{ineq:firstpassI}
\end{equation}
with the second step using exchangeability and the definition of $M$ from \eqref{eq:square_integrability}.
In case (2) of the lemma, when we have the assumption \eqref{assmp:H3}, we take a more refined approach.
Start by expanding the square and using exchangeability to obtain 
\begin{equation*}
I = \frac{k-1}{(n-1)^2}I_{\mathrm{diag}} + \frac{(k-1)(k-2)}{(n-1)^2}I_{\mathrm{cross}},
\end{equation*}
where we define $I_{\mathrm{diag}}$ and $I_{\mathrm{cross}}$ by
\begin{align*}
I_{\mathrm{diag}} &:= \int_{(\R^d)^2} \big| b(t,x_1,x_2) - \langle\mu_t, b(t, x_1, \cdot) \rangle \big|^2 P^2_t(dx), \\
I_{\mathrm{cross}} &:= \int_{(\R^d)^3} \big( b(t,x_1,x_2) - \langle\mu_t, b(t, x_1, \cdot) \rangle\big) \cdot \big (b(t,x_1,x_3) - \langle\mu_t, b(t, x_1, \cdot) \rangle\big) P^3_t(dx).
\end{align*}
We have immediately $I_{\mathrm{diag}} \le M$.
Writing $P^3_t(dx, dx_3) = P_t^2(dx) P_{t,x}^{3|2}(dx_3)$, we have
\begin{align*}
I_{\mathrm{cross}} & = \int_{(\R^d)^2} \big( b(t,x_1,x_2) - \langle\mu_t, b(t, x_1, \cdot) \rangle\big) \cdot \left\langle P_{t,x}^{3|2} - \mu_t ,  b(t,x_1,\cdot) \right\rangle P^2_t(dx) \\
& \leq \sqrt{M} \bigg(\int_{(\R^d)^2} \left| \left\langle P_{t,x}^{3|2} - \mu_t ,  b(t,x_1,\cdot)  \right\rangle \right|^2 P^2_t(dx_1, dx_2)\bigg)^{1/2} \\
& \leq \sqrt{\gamma M} \bigg(\int_{(\R^d)^2} H(P_{t,x}^{3|2} \,|\, \mu_t) \,P_t^2(dx) \bigg)^{1/2} \\
& \leq \sqrt{\gamma M H_T^3}.
\end{align*}
The first inequality comes from Cauchy-Schwarz, the second from the assumption \eqref{eq:transport_type_inequality}, and the last from the chain rule for relative entropy. Using the assumption \eqref{assmp:H3} on $H_T^3$, we thus obtain 
\begin{align*}
I & \leq \frac{k-1}{(n-1)^2} M + \sqrt{\gamma M C_1} \frac{(k-1)(k-2)}{(n-1)^2n^{p/2}}.
\end{align*}
Using $(k-1)/(n-1) \le k/n$, $n-1\ge n/2$, and $p \leq 2$, we have 
\begin{align}
I & \le 2M\frac{k}{n^2} + \sqrt{\gamma M C_1} \frac{k^2}{n^{2+\frac{p}{2}}}  \le C_2 \frac{ k^{2-p/2} }{ n^2 },  \label{eq:first_bound_drift_first_term}
\end{align}
where we recall that $C_2 = 2 M + \sqrt{\gamma M C_1}$. The rest of the proof is given for case (2), but recalling \eqref{ineq:firstpassI} we note that the same proof covers case (1) as long as we set $p=0$ and $C_2=M$.

Using these bounds on $I$ and $II$, the strategy in this proof will be to bound $H_t^k$ in terms of $H_t^{k+1}$, and we will need a bound on the final term $H_t^n$ of the iteration. We thus first apply \eqref{eq:applying_estimate} with $k=n$ and bound the first term of the right-hand side using \eqref{eq:first_bound_drift_first_term} (recalling that $II=0$ for $k=n$). We obtain
\begin{align*}
H_t^n - H_s^n \leq \int_s^t n^{1-p/2}\sigma^{-2} C_2 du - \frac{\sigma^2}{4\LSI }\int_s^t H_u^n du,
\end{align*}
Using Gronwall's inequality (see Lemma \ref{le:Gronwall_adapted} for a precise form tailored to our case), and setting $Z = \sigma^2/4\LSI $, we have 
\begin{align*}
H_t^n & \leq e^{-Zt} H_0^n + \int_0^t n^{1-p/2} C_2 e^{Z(s-t)}ds  \leq C_0 n^{p_1 -p_2} e^{-Zt} + n^{1-p/2}\sigma^{-2}C_2/Z.
\end{align*}

Let us now consider the case $1\le k<n$ and bound $H^k$ in terms of $H^{k+1}$. Returning to \eqref{eq:applying_estimate} and using \eqref{eq:first_bound_drift_first_term} and \eqref{eq:first_bound_drift_second_term}, we have, for $t > s \ge 0$,
\begin{align*}
H^k_t - H^k_s \le \int_s^t \bigg[\frac{(1+\delta^{-1})k^{3-p/2}}{\sigma^{2}n^2}& C_2 + \frac{\gamma k(1+\delta )}{\sigma^2}\big(H^{k+1}_u - H^k_u\big) - \frac{\sigma^2}{4\LSI }H^k_u\bigg]\,du.
\end{align*}
The two $H^k_u$ terms combine to become $-(Z+\tilde{\gamma} k)H^k_u$, where we recall $\tilde{\gamma}=\gamma(1+\delta)/\sigma^2$.
Simplify further by noting that $\frac{1 + \delta^{-1}}{\sigma^2 } = \frac{\tilde{\gamma}}{\delta \gamma}$. 
Apply Gronwall's inequalty (again see Lemma \ref{le:Gronwall_adapted} for a precise form) to get
\begin{align*}
H_t^k & \leq  e^{-(Z + \tilde{\gamma} k)t} H_0^k + \tilde{\gamma}\int_0^t e^{-(Z + k \tilde{\gamma}) (t-s)} \left( \frac{ k^{3-p/2}}{\delta\gamma  n^2} C_2 +  k H_s^{k+1} \right) ds,
\end{align*}Complete the proof via an elementary induction, iterating this inequality for $k+1, k+2, \ldots, n-1$, and then using the assumption on $H_0^k$. 
\end{proof}

\subsection{Estimating the iterated integrals}

Here we prove two lemmas that allow us to estimate the iterated integrals \eqref{eq:def_overline_A_k_ell} and \eqref{eq:def_tilde_B_k_ell} appearing in Lemma \ref{le:first_entropy_estimate}.

\begin{lemma}\label{le:estimate_on_B_tilde}
For any real $p,T \ge 0$ and integers $1 \le k < n$, 
\begin{align}
\sum_{\ell=k}^{n-1} \ell^p \tilde{B}_k^\ell(T) & \le 2k^p(1+p)^p e^{(\tilde{\gamma}p - Z) T}.
\label{eq:bound_for_sum_of_btilde}
\end{align}
\end{lemma}
\begin{proof}
First, we introduce a related iterated integral in which $Z$ is removed from the exponent:
\begin{align}
B _k^\ell(t_k) & = \bigg( \prod_{j=k}^{\ell-1} \tilde{\gamma} j \bigg) \int_0^{t_k} \int_0^{t_{k+1}} \cdots \int_0^{t_{\ell-1}} e^{- \tilde{\gamma} \ell t_\ell -  \sum_{j=k}^{\ell-1} \tilde{\gamma}j (t_{j} - t_{j+1})} dt_{\ell} \ldots dt_{k+1},
\label{eq:B_coeff}
\end{align}
for $t_k \ge 0$.
Recognizing a telescoping sum in the exponent of \eqref{eq:def_tilde_B_k_ell}, we find $\tilde{B}_k^\ell(T) = e^{-ZT}  B_k^\ell(T)$. We thus focus on $B_k^\ell(T)$, which appeared also in \cite[page 25]{LackerFoundation}. Recognize first that $\tilde{\gamma}\ell B_k^\ell = h_k * h_{k+1} * \cdots * h_\ell$, where $h_j(t)=\tilde{\gamma}j e^{-\tilde{\gamma} j t}1_{[0,\infty)}(t)$  is the exponential density with parameters $\tilde{\gamma}j$. Using the expression for the convolution found in \cite[Proof of Lemma 5.2]{LackerFoundation} (applied with $a = \tilde{\gamma} k$ and $b = \tilde{\gamma} $), we obtain 
\begin{align*}
\tilde{\gamma}(k+\ell) B_k^{k+\ell}(t) = h_k * h_{k+1} * \cdots * h_{k+\ell}(t) = \tilde{\gamma} e^{-\tilde{\gamma} kt} \frac{(k+\ell)!}{\ell! (k-1)!}(1 - e^{-\tilde{\gamma} t})^\ell,
\end{align*}
We can then rewrite 
\begin{align}
	\sum_{\ell=k}^\infty \ell^p B_k^{\ell}(t) & = \sum_{\ell = 0}^\infty (k + \ell)^p B_k^{k+\ell}(t) = \frac{e^{-\tilde{\gamma}kt}}{(k-1)!}  \sum_{\ell = 0}^\infty  (k + \ell)^{p-1} \frac{(k+\ell)!}{\ell!}(1 - e^{-\tilde{\gamma}t})^\ell. \label{eq:sum_lp_B_first}
\end{align}

To proceed, we will use the inequality
\begin{align}
(k + \ell)^{p-1} (k+\ell)! \le 2\Gamma(k + \ell + p), \label{eq:Gammabound}
\end{align}
where $\Gamma(z) = \int_0^\infty e^{-t} t^{z-1} dt$ is the gamma function. To prove \eqref{eq:Gammabound}, we set $z = k+\ell$, so that $(k + \ell)^{p-1}(k+\ell)! = z^{p-1} \Gamma(z + 1) = z^p\Gamma(z)$.
In the case $p \ge 1$, a well known gamma ratio inequality \cite[Theorem 1, (5)]{Jameson2013} states that $z^p \Gamma(z) \leq \Gamma(z+p)$, which gives \eqref{eq:Gammabound} without the factor of 2. In the case $0 \le p < 1$, Gautschi's inequality \cite[inequality (7)]{Gautschi1959} states that $(z+1)^{p-1} \le \Gamma(z+p)/\Gamma(z+1)$, which implies
\begin{align*}
(k + \ell)^{p-1} (k+\ell)!=z^{p-1} \Gamma(z+1) \leq \bigg(\frac{z+1}{z}\bigg)^{1-p} \Gamma(z+p) \le 2^{1-p} \Gamma(z+p),
\end{align*}
where the last inequality uses $z \ge 1$.

For $0 < y \le 1$, note by Fubini's theorem and the definition of the gamma function that
\begin{align*}
\sum_{\ell = 0}^\infty \frac{\Gamma(k + \ell + p )}{\ell!}(1-y)^\ell &= \int_0^\infty \sum_{\ell = 0}^\infty \frac{1}{\ell!} u^{k+\ell+p-1} e^{-u} (1-y)^\ell \, du = \int_0^\infty u^{k+p-1} e^{-yu}\,du \\
	&= y^{-(k+p)} \Gamma(k+p).
\end{align*}
Applying this with $y= e^{-\tilde{\gamma}t}$ and then \eqref{eq:Gammabound} in \eqref{eq:sum_lp_B_first} gives 
\begin{align*}
\sum_{\ell=k}^\infty \ell^p B_k^{\ell}(t) &\le \frac{2e^{-\tilde{\gamma}kt}}{(k-1)!}  \sum_{\ell = 0}^\infty   \frac{\Gamma(k+\ell+p)}{\ell!}(1 - e^{-\tilde{\gamma}t})^\ell = 2e^{\tilde{\gamma}rt}\frac{\Gamma(k + p)}{(k-1)!}.
\end{align*}
Finally,  another gamma ratio inequality 
\cite[Theorem 2, (12)]{Jameson2013}
yields $\Gamma(k+p) \le \Gamma(k)k(k+p)^{p-1} = k!(k+p)^{p-1}$, which implies
\begin{equation*}
\frac{\Gamma(k+p)}{(k-1)!} \le k(k+p)^{p-1} \le (k+p)^p \le k^p(1+p)^p. \qedhere
\end{equation*}
\end{proof}

\begin{lemma}\label{le:estimate_on_A_tilde}
Denote $\alpha = \frac{Z}{\tilde{\gamma}}= \frac{\sigma^4}{4\gamma \LSI (1+\delta)}$.
Then, for integers $\ell \ge k \ge 1$,
\begin{align}
\sup_{t\geq 0}  \tilde{A}_k^{\ell} (t) = \prod_{i=k}^\ell \frac{i}{i + \alpha} \leq \left( \frac{k+\alpha}{\ell + 1 + \alpha}\right)^\alpha \le (1+\alpha)^\alpha \left( \frac{k}{\ell + 1}\right)^\alpha, \quad 1 \leq k \leq \ell. 
\label{eq:uniform_bound_A_k_n}
\end{align}
Moreover, for any $p > 0$ with $p \neq \alpha-1$, 
\begin{align}
\left(1 + \frac{(1+\alpha)^\alpha}{|p-\alpha + 1|}\right)^{-1}\frac{1}{n^2}\sum_{\ell = k}^{n-1} \ell^p \sup_{t \geq 0} \tilde{A}_k^\ell(t) & \leq \begin{cases}
                k^\alpha / n^{\alpha + 1 - p} & \text{if } p - \alpha > -1, \\
               k^{p+1}/n^2 & \text{if } p - \alpha < -1.
\end{cases}\label{eq:estimate_sum_ell_A_k_ell}
\end{align} 
\end{lemma}
\begin{proof}
Let $\tilde{h}_j(t)= \tilde{\gamma}j e^{-(Z+\tilde{\gamma}j) t}1_{[0,\infty)}(t)$ for each $j \in \N$.
Note that $\tilde{A}_k^\ell$ and $\tilde{B}_k^\ell$ can be expressed as convolutions,
\begin{align*}
\tilde{A}_k^\ell = \tilde{h}_k * \cdots * \tilde{h}_\ell * 1_{[0,\infty)}, \quad \tilde{\gamma}\ell\tilde{B}_k^\ell = \tilde{h}_k * \cdots \tilde{h}_\ell.
\end{align*}
It follows that $d\tilde{A}_k^\ell/dt = \tilde{\gamma} \ell \tilde{B}_k^\ell$. We observed in \eqref{eq:B_coeff} and just after that $\tilde{\gamma}\ell \tilde{B}_k^\ell(t) = e^{-Zt} h_k * \cdots * h_\ell(t)$, where $h_j(t)= \tilde{\gamma}j e^{- \tilde{\gamma}j t}1_{[0,\infty)}(t)$ is an exponential probability density for each $j \in \N$. 
In other words,
\begin{align*}
\frac{d}{dt} \tilde{A}_k^\ell (t)  = e^{-Zt}h_k * \cdots * h_\ell(t).
\end{align*}
In particuar, $\tilde{A}_k^\ell$ is increasing.
Consider independent exponential random variables $(E_i)_{i\in\N}$ with parameters $\tilde{\gamma} i$. The density of the sum $E_k + \ldots + E_\ell$ is precisely $h_k * \cdots * h_\ell$.
By integrating from $0$ and $t$, we compute
\begin{align*}
\sup_{t\geq 0} \tilde{A}_k^\ell (t) &= \int_0^\infty e^{-Zt} h_k * \cdots * h_\ell(t)\,dt  = \E\left[e^{-Z\left(E_k + \ldots + E_\ell\right)}\right] = \prod_{i=k}^\ell \frac{\tilde{\gamma} i}{\tilde{\gamma} i + Z},
\end{align*}
which proves the first identity in \eqref{eq:uniform_bound_A_k_n}.
Use $\log x\le x - 1$ for $x > 0$ to estimate
\begin{align*}
\sum_{i=k}^\ell \log\frac{ i}{ i + \alpha} \le - \sum_{i=k}^\ell  \frac{\alpha}{i+\alpha} \le -\int_k^{\ell+1} \frac{\alpha}{x+\alpha}\,dx = \alpha\log\frac{k+\alpha}{\ell+1+\alpha},
\end{align*}
and exponentiate to get the first inequality in \eqref{eq:uniform_bound_A_k_n}. The last inequality of \eqref{eq:uniform_bound_A_k_n} follows from $\alpha > 0$ and $k+\alpha\le k(1+\alpha)$.
To prove \eqref{eq:estimate_sum_ell_A_k_ell}, consider two cases:

\begin{enumerate}
\item If $p - \alpha > -1$, we can bound $\sum_{\ell = k}^{n-1} \ell^{p-\alpha} \leq \int_0^{n} x^{p-\alpha}dx = n^{p - \alpha + 1} / (p - \alpha + 1)$.  Thus, 
\begin{align*}
\frac{1}{n^2}\sum_{\ell = k}^{n-1} \ell^p \sup_{t \geq 0} \tilde{A}_k^\ell(t) \leq \frac{(1 + \alpha)^\alpha k^\alpha }{n^2} \sum_{\ell = k}^{n-1} \ell^{p-\alpha} \leq \frac{(1 + \alpha)^\alpha}{p - \alpha + 1} \frac{k^\alpha}{n^{1 +\alpha -p}}.
\end{align*}
\item For $p - \alpha < -1$, we can bound $\sum_{\ell = k+1}^{n-1} \ell^{p-\alpha} \leq \int_k^\infty x^{p-\alpha}dx = k^{p - \alpha + 1} / (\alpha -p - 1)$. Then,
\begin{align*}
\frac{1}{n^2}\sum_{\ell = k}^{n-1} \ell^p \sup_{t \geq 0} \tilde{A}_k^\ell(t) & = \frac{k^p}{n^2} \sup_{t \geq 0} \tilde{A}_k^k(t) + \frac{1}{n^2}\sum_{\ell = k+1}^{n-1} \ell^p \sup_{t \geq 0} \tilde{A}_k^\ell(t) \\
& \leq \frac{k^p}{n^2} \frac{k}{k+\alpha} + \frac{(1 + \alpha)^\alpha k^\alpha}{n^2}\sum_{\ell = k+1}^{n-1} \ell^{p-\alpha}\\
& \leq \frac{k^{p+1}}{n^2} + \frac{(1 + \alpha)^\alpha}{\alpha -p -1} \frac{k^{p+1}}{n^2}.
\end{align*}
\end{enumerate}
{\ } \vskip-0.85cm
\end{proof}

\subsection{Proof of Theorem \ref{th:main-theorem}, for $r_c > 2$} \label{se:pf-part(1)}
We first apply part (1) of Lemma \ref{le:first_entropy_estimate}.
Combining \eqref{eq:first_estimate_on_H_n} and \eqref{eq:first_estimate_on_H_k} and the assumption $H^k_0 \le C_0(k/n)^2$ yields
\begin{align*}
H_T^k \leq \frac{C_0}{n^{2}} \sum_{\ell=k}^{n-1}\ell^{2} \tilde{B}_k^\ell(T) + \frac{M}{\delta \gamma n^2} \sum_{\ell=k}^{n-1} \ell^{2} \tilde{A}_k^\ell(T) + n \tilde{A}_{k}^{n-1}(T)\left(C_0 e^{-ZT} + 4 \LSI M \sigma^{-4}\right). 
\end{align*}
Applying Lemma \ref{le:estimate_on_A_tilde} for the last term, as well as Lemma \ref{le:estimate_on_B_tilde} (with $p=2$) for the first term, 
\begin{align}
H_T^k \leq 18C_0\frac{k^{2}}{n^{2}} e^{(2\tilde{\gamma} - Z)T}+ \frac{M}{\delta \gamma n^2} \sum_{\ell=k}^{n-1} \ell^{2} \tilde{A}_k^\ell(T) + (1+\alpha)^\alpha \frac{k^\alpha}{n^{\alpha-1}}\left(C_0 e^{-ZT} + 4 \LSI M \sigma^{-4}\right). \label{eq:first_pass_argument}
\end{align}
In the following, we write $C$ to denote a constant which can depend only on $(C_0,\gamma,\sigma,M,\LSI )$, and may change from line to line but never depends on $(n,k,T)$.
Recall that $\alpha = \frac{Z}{\tilde{\gamma}}=\frac{\sigma^4}{4\gamma \LSI (1 + \delta)} = \frac{1+r_c}{1+\delta}$.
Since $r_c > 2$, we may choose $\delta>0$ such that $\alpha  > 3$. Clearly $e^{(2\tilde{\gamma}  - Z)T} = e^{\tilde{\gamma}(2- \alpha)T}\leq 1$. Apply Lemma \ref{le:estimate_on_A_tilde} (with $p=2$) in the second term of \eqref{eq:first_pass_argument}, noting that $2-\alpha < -1$, to get
\begin{align*}
H_T^k \leq C \left(\frac{k^2}{n^2} + \frac{k^3}{n^2} + \frac{k^{\alpha}}{n^{\alpha-1}} \right).	
\end{align*}
Noting that $\alpha > 3$, this yields the suboptimal $H^k_T \le Ck^3/n^2$. This is where part (2) of Lemma \ref{le:first_entropy_estimate} enters the picture, allows us to sharpen the $k$ exponent.

Specifically, since we now know $H_T^3 \leq C / n^2$, we can apply part (2) of Lemma \ref{le:first_entropy_estimate},  with $p = 2$, to get $H_T^n \leq C$ and 
\begin{align*}
H_T^k & \leq  \frac{C_0}{n^2}\sum_{\ell=k}^{n-1} \ell^2 \tilde{B}_k^\ell(T) + \frac{C}{  n^2} \sum_{\ell=k}^{n-1} \ell \tilde{A}_k^\ell(T) + \tilde{A}_{k}^{n-1}(T) H_T^n.
\end{align*}
The first term is bounded by $18C_0 (k/n)^2$, by  Lemma \ref{le:estimate_on_B_tilde}. For the third term, we use Lemma \ref{le:estimate_on_A_tilde} to get $\tilde{A}_{k}^{n-1}(T) H_T^n \le C  (k/n)^\alpha \le C(k/n)^2$ since $\alpha > 2$.  For the middle term, we use Lemma \ref{le:estimate_on_A_tilde} (with $p=1$), noting that $1-\alpha < -1$:
\begin{align*}
\frac{1}{n^2} \sum_{\ell=k}^{n-1} \ell \tilde{A}_k^\ell(T) \leq \left(1 + \frac{(1+\alpha)^\alpha}{\alpha-3}\right)(k/n)^2.
\end{align*}
This completes the proof of Theorem \ref{th:main-theorem} in the case $r_c > 2$. \hfill\qedsymbol

\subsection{An inductive lemma} \label{se:induction}

In order to prove Theorem \ref{th:main-theorem} in the remaining case $0 < r_c \le 2$, we first state a lemma which inductively applies part (2) of Lemma \ref{le:first_entropy_estimate}, as was described in the paragraph following the statement of the lemma.
In the following, define 
\begin{align*}
r_c:= \frac{\sigma^4}{4\gamma \LSI }-1, \quad \text{and} \quad q_m:=2(1-2^{-m}), \text{ for } m \in \N,
\end{align*}
and note that $q_{m+1} = 1 + q_m/2$.

\begin{lemma}\label{le:iteration_Lemma}
Let $0 <r < 2$ satisfy $r < r_c$. Suppose $m \in \N$ is such that $r q_m < 2$. Suppose that there exists a constant $C_0$ such that $H_0^k \leq C_0 k^{1 + r}/n^{\min(2, 2r)}$, for all $k=1, \ldots, n$. Suppose that $\sup_{T \ge 0} H_T^3 \leq C_m n^{-rq_m}$ for some constant $C_m>0$. Then, there exists a constant $C_{m}'>0$ depending only on $(C_m, \sigma, \LSI , \gamma, C_0, M)$ such that, by defining $C_{m+1} = \frac{C_m'}{(r_c-r)|2 - rq_{m+1}| }$, we have:
\begin{enumerate}
\item If $rq_{m+1} < 2$, then $\sup_{T \ge 0} H_T^3 \leq C_{m+1} n^{-r q_{m+1}} $, 
\item If $rq_{m+1} > 2$, then $\sup_{T \ge 0} H_T^3 \leq C_{m+1} n^{-2}$.
\end{enumerate}
\end{lemma}
\begin{proof}
Recall that $\alpha = \frac{Z}{\tilde{\gamma}} = \frac{\sigma^4}{4\gamma \LSI (1+\delta)}=(1+r_c)/(1+\delta)$ where the parameter $\delta > 0$ can be chosen freely.
Let $r^*=\min(2,2r)$. Choose $\delta = (r_c-r)/(1+r)$, so that $\alpha=1+r$.
We apply Lemma \ref{le:first_entropy_estimate} and deduce the existence of a constant $\tilde{C}_m>0$, depending only on $(C_m, C_0, \gamma, \LSI , \sigma, M)$ such that 
\begin{align}
H_T^n &\leq \tilde{C}_m (n^{1+r - r^*} + n^{1- \frac{rq_m}{2} }  ), \label{pf:iteration1} \\
H_T^k & \leq \frac{C_0}{n^{r^*}} \sum_{\ell=k}^{n-1} \ell^{1+r} \tilde{B}_k^\ell(T) + \frac{\tilde{C}_m}{\delta n^2} \sum_{\ell=k}^{n-1} \ell^{2- \frac{rq_m}{2}} \tilde{A}_k^\ell(T) + \tilde{A}_{k}^{n-1}(T) H_T^n,  \ 1\leq k < n, \label{pf:iteration2}
\end{align}
for all $T \ge 0$. 
For the first term of the right hand side of \eqref{pf:iteration2}, Lemma \ref{le:estimate_on_B_tilde} yields
\begin{align*}
\sup_{T \ge 0} \frac{C_0}{n^{r^*}} \sum_{\ell=k}^{n-1} \ell^{1+r} \tilde{B}_k^\ell(T) \leq \sup_{T \ge 0} 2 C_0 \frac{k^{1+r}}{n^{r^*}} (2+r)^{1+r} e^{(\tilde\gamma (1+r) - Z)T} \leq 128 C_0 \frac{k^{1+r}}{n^{r^*}},
\end{align*}
where the last step used $\tilde\gamma (1+r) - Z = 0$, and also $r \leq 2$ so that $2(2+r)^{1+r} \le 128$.
Since $r \le 2$, we have $1/\delta \le 3/(r_c - r)$, and so
\begin{align}
H_T^k & \leq 128 C_0 \frac{k^{1+r}}{n^{r^*}} +  \frac{3 \tilde{C}_m}{(r_c-r) n^2} \sum_{\ell=k}^{n-1} \ell^{2- \frac{r q_m}{2}} \tilde{A}_k^\ell(T) + \tilde{A}_{k}^{n-1}(T) H_T^n. \label{pf:iteration-new2}
\end{align}
We next wish to combine the exponents.

Suppose first that $r q_{m+1} < 2$. Then $2 - \frac{rq_m}{2} - \alpha > -1$, and Lemma \ref{le:estimate_on_A_tilde} with $\alpha=1+r$ and $q=2 - \frac{rq_m}{2}$ yields 
\begin{align*}
H_T^k & \leq 128 C_0 \frac{k^{1+r}}{n^{r^*}} + \frac{3\tilde{C}_m}{r_c-r}\left(1 + \frac{(2+r)^{1+r}}{2 - rq_m/2 - r} \right) \frac{k^{1+r}}{n^{2 + r - (2 - r q_m/2)}} + (2+r)^{1+r} \frac{k^{1+r}}{n^{1+r}}H_T^n\\
&\le 128 C_0 \frac{k^{1+r}}{n^{r^*}} + \frac{3\tilde{C}_m}{r_c-r}\left(1 + \frac{64}{2 - rq_{m+1}} \right) \frac{k^{1+r}}{n^{rq_{m+1}}} + 64\frac{k^{1+r}}{n^{1+r}}H_T^n,
\end{align*}
where we used $q_{m+1} = 1 + q_m/2$ and $(2+r)^{1+r} \le 64$, the latter following from $r \le 2$. 
To simplify this further, we claim first that
\begin{equation}
1+r-r^* < 1 - \frac{rq_m}{2}. \label{pf:iteration-new1}
\end{equation}
Indeed, if $r>1$, then $r^* = 2$, and \eqref{pf:iteration-new1} follows from the assumed inequality $r+rq_m/2 = rq_{m+1} < 2$. If instead $r \leq 1$, then $r^* = 2r$, and \eqref{pf:iteration-new1} follows simply from $q_m < 2$.
Now, using \eqref{pf:iteration-new1}, we see that  \eqref{pf:iteration1} implies $H_T^n \leq 2\tilde{C}_m n^{1 - \frac{rq_m}{2}}$. Using again $q_{m+1} = 1 + q_m/2$, we deduce that $H_T^n/ n^{1+r} \le 2\tilde{C}_m / n^{ rq_{m+1}}$. 
Also, rearranging \eqref{pf:iteration-new1} shows that $r^* \ge rq_{m+1}$, so $n^{-r^*} \le n^{-rq_{m+1}}$.
We can now simplify the first and last terms in \eqref{pf:iteration-new2} to get
\begin{align*}
H_T^k & \le 128 C_0 \frac{k^{1+r}}{n^{rq_{m+1}}} + \frac{3\tilde{C}_m}{r_c-r}\left(1 + \frac{64}{2 - rq_{m+1}} \right) \frac{k^{1+r}}{n^{rq_{m+1}}} + 128\tilde{C}_m\frac{k^{1+r}}{n^{rq_{m+1}}}.
\end{align*}
Evaluating this for $k=3$ and using $1+r \le 3$ gives 
\begin{align*}
H_T^3 n^{rq_{m+1}} & \leq 3456 C_0 + \frac{81\tilde{C}_m}{r_c-r}\left(1 + \frac{64}{2 - rq_{m+1}} \right) + 3456 \tilde{C}_m,
\end{align*}
Using $r_c-r \le r_c$ and $2-rq_{m+1} \le 2$, we may put with each term the same denominator  $(r_c-r)(2-rq_{m+1})$. This completes the proof in the case $rq_{m+1} < 2$.

If instead $r q_{m+1} > 2$, then $r >1$ because $q_{m+1} < 2$, and thus $r^* = 2$.  Moreover, $1 - \frac{r q_m}{2} < r -1$, and thus $H_T^n \leq 2 \tilde{C}_m n^{r-1}$ by \eqref{pf:iteration1}. Apply Lemma \ref{le:estimate_on_A_tilde} to obtain this time
\begin{align*}
H_T^k & \leq 128 C_0 \frac{k^{1+r}}{n^{2}} + \frac{3\tilde{C}_m}{ r_c-r }\left(1 + \frac{(2+r)^{1+r}}{rq_{m+1} - 2} \right) \frac{k^{3 - \frac{r q_m}{2}}}{n^2} + (2+r)^{1+r} \frac{k^{1+r}}{n^{1+r}}2\tilde{C}_m n^{r-1}.
\end{align*}
Apply this with $k=3$ to complete the proof.
\end{proof}

\subsection{Proof of Theorem \ref{th:main-theorem}, for $1 < r_c \le 2$} \label{se:pf-part(1)-2}
We begin exactly as in Section \ref{se:pf-part(1)}, applying  part (1) of Lemma \ref{le:first_entropy_estimate} to deduce the inequality \eqref{eq:first_pass_argument}.
We will again use $C$ to denote a constant which depends only on $(C_0,\gamma,\sigma,M,\LSI )$, and may change from line to line, but never depends on $(n,k,T)$.
Recall that $\alpha = \frac{Z}{\tilde{\gamma}} =\frac{\sigma^4}{4\gamma \LSI (1 + \delta)}=\frac{1+r_c}{1+\delta}$, where $\delta > 0$ can be chosen freely. We choose here $\delta$ so that $2 < \alpha < 3$ and such that $r := \alpha - 1 \notin \{2/q_m : m \in \N\}$. Note then that $1 < r < r_c \le 2$.

We simplify \eqref{eq:first_pass_argument} in this case as follows.
For the first term, note that $e^{(2\tilde{\gamma} - Z)T} = e^{\tilde{\gamma}(2-\alpha)T}\leq 1$.  For the second and third terms, apply Lemma \ref{le:estimate_on_A_tilde} (with $p=2$), noting that $2 - \alpha > -1$:
\begin{align*}
H_T^k \le C \left( \frac{k^2}{n^2} + \frac{k^{r+1}}{n^{r}} +   \frac{k^{r+1} }{n^{r}} \right).
\end{align*}
Setting $k=3$ yields $H_T^3 \leq C / n^r = C / n^{ r q_1}$.
The idea is to now repeatedly apply Lemma \ref{le:iteration_Lemma}, $m^*-1$ times, where $m^*$ is the smallest integer such that $r q_{m^*} > 2$, which exists because $\lim_{m\to\infty} r q_m = 2 r > 2$. Note $r q_m \neq 2$ for all $m$, by design. This repeated application of  Lemma \ref{le:iteration_Lemma} ultimately yields $H_T^3 \leq C /n^2$.
We may now conclude exactly as in the last paragraph of Section \ref{se:pf-part(1)}, noting that $\alpha > 2$. \hfill\qedsymbol

\subsection{Proof of Theorem \ref{th:main-theorem}, for $0 < r_c \le 1$} \label{se:pf-part(2)}
Let $0<\epsilon_1 < \epsilon_2 < r_c$. We first claim that we can find a constant $C'$ depending only on $(\gamma,\sigma,M,\LSI,\epsilon_1,\epsilon_2,  C_0^{(\epsilon_2-\epsilon_1)/2, \epsilon_2-\epsilon_1})$, such that
\begin{align}
\sup_{T \ge 0} H^3_T \le C' / n^{2(r_c-(\epsilon_2-\epsilon_1))}. \label{pf:case3-Hk}
\end{align}
Abbreviate $\epsilon:=\epsilon_2-\epsilon_1$ and $r := r_c - \epsilon/2$, and  choose $\delta>0$  such that $\alpha = 1 + r \in (1,2)$. 
Let $p_1=1 + r_c - \epsilon/2$ and $p_2 = 2r_c - \epsilon$.
We have by assumption $H(P_0^k \,|\, \mu_0^{\otimes k }) \le  C_0^{\epsilon/2, \epsilon } k^{p_1} /n^{p_2}$.
Apply Lemma \ref{le:first_entropy_estimate}(1), plugging \eqref{eq:first_estimate_on_H_n} into \eqref{eq:first_estimate_on_H_k} and using $p_1-p_2\le 1$, we find
\begin{align*}
H_T^k \leq \frac{C_0^{\epsilon/2, \epsilon}}{n^{p_2}} \sum_{\ell=k}^{n-1}\ell^{p_1} \tilde{B}_k^\ell(T) + \frac{M}{\delta \gamma n^2} \sum_{\ell=k}^{n-1} \ell^{2} \tilde{A}_k^\ell(T) + n \tilde{A}_{k}^{n-1}(T)\left(C_0 e^{-ZT} + 4 \LSI M \sigma^{-4}\right). 
\end{align*}
Apply Lemma \ref{le:estimate_on_A_tilde} for the last term, as well as Lemma \ref{le:estimate_on_B_tilde} for the first term:
\begin{align}
H_T^k \leq 18 C_0^{\epsilon/2, \epsilon}\frac{k^{p_1}}{n^{p_2}} e^{(\tilde{\gamma} p_1 - Z)T}+ \frac{M}{\delta \gamma n^2} \sum_{\ell=k}^{n-1} \ell^{2} \tilde{A}_k^\ell(T) + (1+\alpha)^\alpha \frac{k^\alpha}{n^{\alpha-1}}\left(C_0 e^{-ZT} + 4 \LSI M \sigma^{-4}\right), \label{eq:first_pass_argument-lastcase1}
\end{align}
where we used $p_1 \le 2$ to get $2(1+p_1)^{p_1} \le 18$.
Simplify the first term by noting that $p_1= \alpha$ and so $e^{(\tilde{\gamma}p_1 - Z)T} = 1$. Apply Lemma \ref{le:estimate_on_A_tilde} (with $p=2$) to the second term in \eqref{eq:first_pass_argument-lastcase1}, noting that $2 - \alpha > -1$, to get
\begin{align*}
H_T^k \leq C'\left(\frac{k^{1+r} }{n^{2r}} +  \frac{k^{1+r}}{n^{r}} +  \frac{k^{1+r}}{n^r} \right),
\end{align*}
with $C'$ depending only on $(\gamma,\sigma,M,\LSI,\epsilon,  C_0^{\epsilon/2, \epsilon})$.

By evaluating the above for $k=3$, we have $H_T^3 \leq C'/n^r$. 
We may now repeatedly apply Lemma \ref{le:iteration_Lemma}, as in Section \ref{se:pf-part(1)-2}, but this time we never reach the case where $rq_m \ge 2$ because $rq_m \uparrow 2r < 2r_c \le 2$ as $m\to\infty$. Instead, we choose $m$ large enough so that $rq_m \ge 2r - \epsilon$ and deduce that $H^3_T \le C/n^{2r-\epsilon}=C/n^{2r_c-2\epsilon}$, where $C>0$ now depends only on $(\gamma,\sigma,M,\LSI,\epsilon_1,\epsilon_2,  C_0^{\epsilon/2, \epsilon})$. For the rest of this proof, $C>0$ may additionally depend on $C_0^{\epsilon_1, \epsilon_2}$ and may change from line to line, but it will never depend on $(n,k,T)$.

With \eqref{pf:case3-Hk} now established, let us now instead choose $\delta>0$ (depending only on $\sigma, \gamma, \LSI, \epsilon_1$) so that $\alpha=1+r_c-\epsilon_1$, and define $p_1=1 + r_c - \epsilon_1$ and $p_2=2r_c - \epsilon_2$. Then \eqref{eq:first_pass_argument-lastcase1} again holds but with 
$C_0^{\epsilon/2,\epsilon}$ replaced by $C_0^{\epsilon_1,\epsilon_2}$.
In light of \eqref{pf:case3-Hk}, we may apply  Lemma \ref{le:first_entropy_estimate}(2) with $p=2(r_c-(\epsilon_2-\epsilon_1))$ to get
\begin{align*}
H_T^n &\le C n^{1-r_c+\epsilon_2-\epsilon_1}, \\
H^k_T &\le C \frac{k^{1 + r_c - \epsilon_1 } }{n^{2r_c - \epsilon_2}} + \frac{C}{n^2}\sum_{\ell=k}^{n-1} \ell^{2 - r_c + \epsilon_2-\epsilon_1 }\tilde{A}_k^\ell(T) + \tilde{A}_k^{n-1}(T)H_T^n.
\end{align*}
Apply Lemma \ref{le:estimate_on_A_tilde} with $p=2 - r_c + \epsilon_2-\epsilon_1$, noting that then $p-\alpha = 1 - 2r_c + \epsilon_2 > -1$, to get
\begin{equation*}
\frac{1}{n^2}\sum_{\ell=k}^{n-1} \ell^{2 - r_c + \epsilon_2-\epsilon_1 }\tilde{A}_k^\ell(T) \le C \frac{k^{\alpha}}{n^{\alpha+1-p}} = C \frac{k^{1+r_c-\epsilon_1}}{n^{2r_c-\epsilon_2 }}.
\end{equation*}
Lemma \ref{le:estimate_on_A_tilde} also yields
\begin{equation*}
\tilde{A}_k^{n-1}(T)H_T^n \le C \frac{k^\alpha}{n^\alpha}  n^{1-r_c+\epsilon_2-\epsilon_1} = C \frac{k^{1+r_c-\epsilon_1}}{n^{2r_c-\epsilon_2 }}.
\end{equation*}
Put it together to complete the proof of this final case, and thus the theorem. \hfill\qedsymbol

\section{Proof of for reversed entropy}\label{se:sec-proof-reversed-entropy}
In this section we prove Theorem \ref{th:reverse}.
Define the reversed relative entropy
\begin{equation*}
H^k_t := H(\mu^{\otimes k}_t\,|\,P^k_t), \quad t \ge 0, \ 1 \le k \le n.
\end{equation*}
By assumption, $H_0^k \leq C_0 (k/n)^p$ for some $C_0$ and $0 < p \le 2$, with the choice of $p$ depending on which case in Theorem \ref{th:reverse} we are considering. We make use again of the functions $\widehat{b}^k_i$ defined in \eqref{def:bhat^k_i}.
We first apply Lemma \ref{le:entropy-estimate} to compute the time-derivative of $H^k_t$. It is evident that (H.1,2) hold because $b$ is bounded and $b_0$ is locally bounded. Moreover, (H.3) is a direct consequence of \eqref{eq:integ-b_0-rev}.
We thus obtain, for all $t > s \ge 0$,
\begin{align*}
H_t^k - H_s^k + \frac{\sigma^2}{4} \int_s^t I(\mu_u^{\otimes{k}} | P_u^k)du & \leq \frac{1}{\sigma^2}\int_s^t \int_{(\R^d)^k} \sum_{i=1}^k\big|\widehat{b}^k_i(u,x) - \langle \mu_u,b(u,x_i,\cdot)\rangle\big|^2 \mu_u^{\otimes{k}}(dx)du.
\end{align*}
The LSI \eqref{eq:LSI-rev} marginalizes, in the sense that
\begin{align*}
H(\nu \,|\, P_t^k) \leq \LSI  I(\nu \,|\, P_t^k), \quad \forall k=1,\ldots,n, \ \nu \in \P((\R^d)^k), \ t \geq 0.
\end{align*}
Indeed, for a given $k$ and $\nu \in \P((\R^d)^k)$, this follows by applying \eqref{eq:LSI-rev} to the measure $\tilde\nu \in \P((\R^d)^n)$ with density $d\tilde\nu/dP^n_t(x_1,\ldots,x_n) = d\nu/dP^k_t(x_1,\ldots,x_k)$.
This yields
\begin{align}
H_t^k - H_s^k + \frac{\sigma^2}{4\LSI } \int_s^t H_u^k du& \leq \frac{1}{\sigma^2}\int_s^t \int_{(\R^d)^k} \sum_{i=1}^k\big|\widehat{b}^k_i(u,x) - \langle \mu_u,b(u,x_i,\cdot)\rangle\big|^2 \mu_u^{\otimes{k}}(dx)du \nonumber  \\
& = \frac{k}{\sigma^2}\int_s^t \int_{(\R^d)^k} \big|\widehat{b}^k_1(u,x) - \langle \mu_u,b(u,x_1,\cdot)\rangle\big|^2 \mu_u^{\otimes{k}}(dx)du , \label{pf:rev:Hmain} 
\end{align}
where the second line comes from exchangeability. 
Use the definition of $\widehat{b}^k_1$ to write the inside of the time integral of right-hand side as
\begin{align*}
\frac{k}{\sigma^2}\int_{(\R^d)^k}\bigg|\frac{1}{n-1}\sum_{j=2}^k \big( b(t,x_1,x_j) - \langle\mu_t, b(t, x_1, \cdot) \rangle\big) + \frac{n-k}{n-1}\big\langle P^{k+1|k}_{t,x} -  \mu_t, b(t, x_1, \cdot) \big\rangle\bigg|^2 \mu_t^{\otimes{k}}(dx).
\end{align*}
Let $\delta > 0$, to be chosen later, and use the inequality $(x+y)^2 \leq (1+\delta)x^2 + (1 + \delta^{-1})y^2$ to further bound this by
\begin{equation*}
\frac{k}{\sigma^2}(1+\delta^{-1})I + \frac{k}{\sigma^2}(1 + \delta)II,
\end{equation*}
where we set
\begin{equation*}
\begin{split}
I &:= \int_{(\R^d)^k}\bigg|\frac{1}{n-1}\sum_{j=2}^k \big( b(t,x_1,x_j) - \langle\mu_t, b(t, x_1, \cdot) \rangle\big) \bigg|^2\mu_t^{\otimes{k}}(dx), \\
II &:= \int_{(\R^d)^k}\left| \frac{n-k}{n-1}\big \langle P^{k+1|k}_{t,x} -  \mu_t, b(t, x_1, \cdot) \big\rangle\right|^2 \mu_t^{\otimes{k}}(dx),
\end{split}
\end{equation*}
again with the convention that $II=0$ for $k=n$.
Expanding the square in $I$, the cross-terms vanish, and by exchangeability we have 
\begin{align}
I = \frac{k-1}{(n-1)^2} \int_{(\R^d)^2} \big| b(t,x_1,x_2) - \langle\mu_t, b(t, x_1, \cdot) \rangle\big|^2 \mu_t^{\otimes{2}}(dx) \leq \frac{ \gamma (k-1)}{2(n-1)^2}, \label{pf:rev:I}
\end{align}
where we set $\gamma:= 2 \| |b|^2 \|_{\infty}$. We note here for future use that when $k=n$, we have $I \le \gamma/2(n-1)$ and $II=0$, so that \eqref{pf:rev:Hmain} becomes
\begin{equation*}
H_t^n - H_s^n + \frac{\sigma^2}{4\LSI } \int_s^t H_u^n du \leq \frac{\gamma n}{2\sigma^2(n-1)} \leq \frac{ \gamma}{\sigma^2}.
\end{equation*}
Using a form of Gronwall's inequality (see Lemma \ref{le:Gronwall_adapted}), and setting $Z = \sigma^2/4\LSI $, we have 
\begin{align}
H_t^n & \leq e^{-Zt} H_0^n + \int_0^t \frac{ \gamma}{\sigma^2} e^{Z(s-t)}ds  \leq C_0 e^{-Zt} + \frac{4  \gamma \LSI }{\sigma^4}. \label{pf:rev:H^n}
\end{align}
Next, for $k < n$, we bound $II$. Since $b$ is bounded, Pinsker's inequality yields 
\begin{align*}
\big|\big \langle P^{k+1|k}_{t,x} -  \mu_t, b(t, x_1, \cdot) \big\rangle\big|^2 \leq \gamma H(\mu_t \,|\, P^{k+1|k}_{t,x}).
\end{align*}
Using $(n-k)/(n-1) \leq 1$ and the chain rule for relative entropy (as in \eqref{eq:first_bound_drift_second_term}), we get
\begin{align*}
II \leq \gamma \int_{(\R^d)^k} H(\mu_t \,|\, P^{k+1|k}_{t,x})\, \mu_t^{\otimes{k}}(dx) = \gamma  \big(H_t^{k+1} - H_t^k\big).
\end{align*}
Adding this to \eqref{pf:rev:I} and returning to \eqref{pf:rev:Hmain} yields
\begin{align*}
H_t^k - H_s^k \leq \int_s^t \left(\frac{ k(k-1) \gamma(1+\delta^{-1})}{2(n-1)^2\sigma^2} + \tilde{\gamma}k H_u^{k+1} - (Z+\tilde{\gamma}k) H_u^k \right)du.
\end{align*}
where we set $\tilde\gamma = \gamma(1+\delta)/\sigma^2$.
The first term inside the integral is bounded by $\tilde\gamma k^2/\delta n^2$, which follows from using $(k-1)/(n-1) \leq k/n$ and $n\le 2(n-1)$ as well as the identity $\frac{1 + \delta^{-1}}{\sigma^2 } \gamma = \frac{\tilde{\gamma}}{\delta}$. By Gronwall's inequality (as in Lemma \ref{le:Gronwall_adapted}), 
\begin{align*}
H_t^k \leq e^{-(Z+\tilde{\gamma}k) t} H_0^k + \int_0^te^{-(Z+\tilde{\gamma}k) s}\left( \frac{ \tilde\gamma k^2 }{\delta n^2} + \tilde{\gamma}k H_s^{k+1} \right)ds.
\end{align*}
By iterating for $\ell = k, k+1, \ldots, n$ as in the proof of Lemma \ref{le:first_entropy_estimate}, we have
\begin{align*}
H_T^k \leq \sum_{\ell=k}^{n-1} H_0^\ell \tilde{B}_k^\ell(T) +\frac{1}{\delta n^2} \sum_{\ell=k}^{n-1} \ell \tilde{A}_k^\ell(T)+ \tilde{A}_{k}^{n-1}(T) H_T^n,
\end{align*}
where $\tilde{A}_{k}^{\ell}$ and $\tilde{B}_{k}^{\ell}$ (and $Z$ and $\tilde{\gamma}$) are defined as in Lemma \ref{le:first_entropy_estimate}. Plugging in \eqref{pf:rev:H^n} and the assumption $H^k_0 \le C_0(k/n)^p$, we obtain
\begin{align}
H_T^k \leq  \frac{C_0}{n^{p}}\sum_{\ell=k}^{n-1} \ell^{p} \tilde{B}_k^\ell(T) +\frac{1}{\delta n^2} \sum_{\ell=k}^{n-1} \ell \tilde{A}_k^\ell(T)+ \tilde{A}_{k}^{n-1}(T) \left(C_0 e^{-Zt} +  4  \gamma \LSI  \sigma^{-4} \right).
\label{eq:sharpest_estimate_on_rev_H_k}
\end{align}
Let $\alpha = \frac{Z}{\tilde{\gamma}} = \frac{\sigma^4}{4\gamma \LSI  ( 1 + \delta)}$ and apply Lemmas \ref{le:estimate_on_B_tilde} and \ref{le:estimate_on_A_tilde} in \eqref{eq:sharpest_estimate_on_rev_H_k} to get
\begin{align}
H_T^k \leq  \frac{2 C_0 k^{p}}{n^{p}} (1+p)^{p} e^{Z(\frac{p}{\alpha}-1)T} +\frac{1}{\delta n^2} \sum_{\ell=k}^{n-1} \ell \tilde{A}_k^\ell(T)+ (1 + \alpha)^\alpha \frac{k^\alpha}{n^\alpha} \left(C_0+ 4 \gamma \LSI  \sigma^{-4}\right).
\label{eq:last-equation-conclusion}
\end{align}
We now complete the proof in two cases separately.

\subsubsection*{First case: $p_c=\frac{\sigma^4}{4\gamma \LSI } > 2$}
In this case $p = 2$. Choose $\delta>0$ so that $\alpha > 2$. Using Lemma \ref{le:estimate_on_A_tilde},
\begin{align*}
\frac{1}{n^2} \sum_{\ell=k}^{n-1} \ell \tilde{A}_k^\ell(T) \leq \left(1 + \frac{(1 + \alpha)^\alpha}{\alpha -2}\right) \frac{k^2}{n^2}.
\end{align*}
Using this in \eqref{eq:last-equation-conclusion} and noting that $(k/n)^\alpha \le (k/n)^2$ completes the proof.

\subsubsection*{Second case:  $p_c=\frac{\sigma^4}{4\gamma \LSI } \leq 2$}
In this case $p = \frac{\sigma^4}{4 \gamma \LSI }-\epsilon$, where  $\epsilon \in (0, p_c)$. Choose $\delta>0$ so that $\alpha = p$. Since $\alpha<2$, applying Lemma \ref{le:estimate_on_A_tilde}  yields
\begin{align*}
\frac{1}{n^2} \sum_{\ell=k}^{n-1} \ell \tilde{A}_k^\ell(T) \leq \left(1 + \frac{(1 + \alpha)^\alpha}{2- \alpha}\right) \frac{k^\alpha}{n^\alpha}.
\end{align*}
Using this in \eqref{eq:last-equation-conclusion} completes the proof. \hfill \qedsymbol

\section{Proof of Corollaries}\label{se:sec-proof-corollaries}

\subsection{Convex potentials: Proof of Corollary \ref{cor:convex-main}}\label{se:proof_of_first_corollary}

Here $b_0(t,x) = - \nabla U(x)$ and $b(t,x,y) = - \nabla W(x-y)$. We need to check that the conditions of Assumptions \ref{assumption:E} and \ref{assumption:A} hold. Note that the LSI for $\mu_0$ is known to imply that $\mu_0$ has finite moments of every order.

We start with Assumption \ref{assumption:E}.
Assumption (\ref{assumption:E}.1) is immediate by continuity of $(\nabla U,\nabla W)$. Suppose that $L<\infty$. The existence and uniqueness of a weak solution $(\mu_t)_{t \ge 0}$ of the nonlinear Fokker-Planck equation follows from \cite[Theorem 2.6]{Malrieu2008}, and since $\nabla W$ Lipschitz we may use the moment bounds of \cite[Theorem 2.6]{Malrieu2008} to deduce $\nabla W(x-\cdot) \in L^1(\mu_t)$ for all $(t,x)$ as well as the continuity and thus local boundedness of $(t,x) \mapsto \langle \mu_t, \nabla W(x-\cdot)\rangle$. This shows Assumption (\ref{assumption:E}.2). 
The existence of a unique (starting from the given initial law $P^n_0$) solution $(P^n_t)_{t \ge 0}$ to the Fokker-Planck equation required by Assumption (\ref{assumption:E}.3) is justified in \cite[Section 2]{Malrieu2008}. 
Because $\nabla U$ and $\nabla W$ have polynomial growth in $x$, the integrability requirements of Assumption (\ref{assumption:E}.3) are a simple consequences of the moments bounds for $P^n_t$ and $\mu_t$ given in  \cite[Corollary 2.3, Theorem 2.6]{Malrieu2008}, which adapt easily to the case of non-i.i.d.\ initial conditions. If $L=\infty$ but $R<\infty$, we verify Assumption \ref{assumption:E} in a similar  fashion. 

We now turn to checking Assumption \ref{assumption:A}. Using the assumption $\nabla^2 (U + W) \succeq \alpha I$, the (time-inhomogeneous) generator $\varphi \mapsto - \nabla \varphi \cdot \nabla U  - \nabla \varphi \cdot \nabla W * \mu_t + \frac{\sigma^2}{2}\Delta\varphi$ of the Fokker-Planck equation satisfied by $(\mu_t)_{t \ge 0}$ obeys the curvature condition described in \cite[Proposition 3.12]{Malrieu2001} and \cite[Theorem 4.1]{Collet2008}, and we may thus deduce from their arguments (see also \cite[Corollary 3.7]{Malrieu2001}) that $\mu_t$ satisfies a LSI with constant $\frac{\sigma^2}{4\alpha}(1-e^{-4\alpha t/\sigma^2}) + \frac{\LSI _0}{4}e^{-4\alpha t/\sigma^2}$. This shows that Assumption (\ref{assumption:A}.1) holds with $\LSI =\max(\LSI _0/4,\sigma^2/4\alpha)$.
For Assumption (\ref{assumption:A}.2), first notice that if $R<\infty$, Pinsker's inequality gives us the required inequality with $\gamma = 2 R^2$. In the case where $L<\infty$, recall from Remark \ref{re:OttoVillani} that $\mu_t$ obeys the quadratic transport inequality \eqref{eq:OttoVillani}.
Hence, since $\nabla W$ is $L$-Lipschitz, we use   Kantorovich duality to deduce
\begin{align}
\big|\langle \nu - \mu_t, \nabla W(x - \cdot)\rangle \big|^2 \leq L^2 \W_1^2(\mu_t, \nu) \leq L^2 \W_2^2(\mu_t, \nu) \leq 4\LSI L^2  H(\nu \,|\, \mu_t).
\end{align}
This yields Assumption (\ref{assumption:A}.2) with $\gamma=4\LSI L^2$.
We finally check that the constant $M$ of Assumption (\ref{assumption:A}.3) is finite and bounded uniformly in $n$, which here is
\begin{align}
M := \esssup_{t \geq 0} \E\left|\nabla W(X_t^1- X_t^2) - \langle \mu_t, \nabla W(X_t^1-\cdot) \rangle \right|^2.
\label{eq:uniform_second_moment_Malrieu}
\end{align}
The case $R<\infty$ is evident. If $L < \infty$ then $\nabla W$ has linear growth, and we conclude from the second moment bounds of \cite[Corollary 2.3, Proposition 2.7]{Malrieu2008}. \hfill\qedsymbol

\subsection{Models on the torus: Proof of Corollary \ref{co:torus}}

The weak well-posedness of the (linear) Fokker-Planck equation \eqref{def:PDEtorus} is standard, as $K$ is bounded and Lipschitz.
The well-posedness of the SDE \eqref{def:SDEtorus} follows by the standard argument of Sznitman \cite[Chapter I.1]{sznitman1991topics}, since $K$ is Lipschitz. The well-posedness of the PDE \eqref{def:MVPDEtorus} in the weak sense follows by the superposition principle \cite[Theorem 2.6]{figalli2008existence}. The fact that the solution is classical follows from standard results on linear Fokker-Planck equations $\partial_t \mu = -\div(\tilde{K} \mu) + (\sigma^2/2)\Delta \mu$, applied with the bounded Lipschitz drift $\tilde{K} = K * \mu$. Strict positivity follows from Harnack's inequality \cite[Chapter 8]{bogachev-book}.

The following lemma paves the way for a short proof of Corollary \ref{co:torus}, given just after:

\begin{lemma} \label{le:torus}
For each $t > 0$, the density of $\mu_t$ is $C^2$ and obeys the pointwise bound 
\begin{align*}
\frac{1}{\lambda e^{r_0}} \le \mu_t(x) \le \frac{\lambda}{1- {r_0} e^{r_0}}, \quad \text{where} \quad r_0 := \frac{\|\div\, K\|_\infty\sqrt{ 2 \log \lambda}}{\sigma^2\pi^2 - \|\div\, K\|_\infty},
\end{align*}
Moreover, it holds that $r_0 < 1/2$, and  $\mu_t$ satisfies the LSI
\begin{align*}
H(\cdot\,|\,\mu_t) \le \LSI I(\cdot\,|\,\mu_t), \quad \text{where } \ \LSI  := \lambda^2 /(1- 2r_0).
\end{align*}
\end{lemma}
\begin{proof}
We note for future use that  \cite[Proposition 3.1]{carrillo2020long} shows
\begin{align}
H(\mu_t\,|\,\bm{1}) \le e^{-2ct}H(\mu_0\,|\,\bm{1}), \ \ \forall t \ge 0, \text{ where } c := \sigma^2\pi^2 - \|\div\, K\|_\infty,  \label{pf:ent-torus1}
\end{align}
and where $\bm{1}$ denotes the uniform measure on $\T^d$.
The precise setting of  \cite[Proposition 3.1]{carrillo2020long} is somewhat different, as they take $K=\nabla W$ as a gradient field and have a sharper estimate involving only the ``unstable part" $\Delta W_u$ in place of $\div \, K$ in the definition of $c$, but their proof works (and becomes slightly simpler) to yield \eqref{pf:ent-torus1}.
We next prove the pointwise bound on $\mu_t$, following an idea from the proof of \cite[Theorem 2]{guillin2021uniform}. Fix $T > 0$ and $x \in \T^d$, and let $(Y_t)_{t \in [0,T]}$ denote the unique strong solution of the SDE
\begin{align*}
dY_t = -K * \mu_{T-t}(Y_t)dt + \sigma dB_t, \quad Y_0=x.
\end{align*}
Using the PDE for $\mu$ and It\^o's formula, then taking expectations,
we have for $t \in [0,T]$
\begin{align}
\E\mu_{T-t}(Y_t) &= \mu_T(x) + \E\int_0^t\mu_{T-s}(Y_s) \div\, K * \mu_{T-s} (Y_s)\,ds. \label{pf:ent-torus2}
\end{align}
Note if $\div\, K \equiv 0$, as in \cite[Theorem 2]{guillin2021uniform}, then this would immediately yield the desired conclusion $\mu_T(x)=\E\mu_0(Y_T) \in [\lambda^{-1},\lambda]$. In general, we need an additional argument. Noting that $\div\, K * \bm{1} \equiv 0$, we have for any $u \in [0,T]$ that
\begin{align*}
\|\div\, K * \mu_u\|_\infty &\le \|\div\, K * (\mu_u-\bm{1})\|_\infty \le \|\div\, K \|_\infty \| \mu_u-\bm{1} \|_{\mathrm{TV}} \le \|\div\, K \|_\infty \sqrt{2H(\mu_u\,|\,\bm{1})} \\
	&\le \sqrt{2  \log \lambda } \|\div\, K \|_\infty e^{-cu},
\end{align*}
with the last step using \eqref{pf:ent-torus1} and $H(\mu_u\,|\,\bm{1}) \le \log \lambda$. Setting $a=\sqrt{2 \log \lambda} \|\div\, K \|_\infty$, this implies
\begin{align*}
\E\mu_{T-t}(Y_t) &\le \mu_T(x) +  a\int_0^t e^{-cs}\E\mu_{T-s}(Y_s)\,ds.
\end{align*}
By Gronwall's inequality, we obtain  for all $t \in [0,T]$
\begin{align*}
\E\mu_{T-t}(Y_t) \le \mu_T(x)\exp\bigg(a\int_0^t e^{-cs}ds\bigg) \le \mu_T(x) e^{a/c}.
\end{align*}
Setting $t=T$ yields $\mu_T(x) \ge e^{-a/c}\E\mu_0(Y_T) \ge e^{-a/c}\lambda^{-1}$, as well as by \eqref{pf:ent-torus2}
\begin{align*}
\mu_T(x) &\le \E\mu_{0}(Y_T) + a \E\int_0^T e^{-cs}\mu_{T-s}(Y_s)  \,ds  \le \E\mu_{0}(Y_T) + \frac{a}{c} \mu_T(x) e^{a/c}.
\end{align*}
Combining the last two inequalities, we find
\begin{align*}
e^{-a/c}\lambda^{-1} \le \mu_T(x) \le \E\mu_0(Y_T) (1-(a/c)e^{a/c})^{-1} \le \lambda (1-(a/c)e^{a/c})^{-1}.
\end{align*}
This yields the claimed bounds on the density $\mu_t$.
The assumption \eqref{asmp:Ksmall} ensures that $r_0 < 1/2$.
Finally, the claimed LSI follows from the Holley-Stroock argument \cite[Proposition 5.1.6]{Bakry},  after noting that  $\sup\mu_t/ \inf \mu_t \le  \lambda^2 e^{r_0}/(1-r_0 e^{r_0}) \le \lambda^2/(1-2r_0)$ since $0 \le r_0 < 1/2$.
\end{proof}

\begin{proof}[Proof of Corollary \ref{co:torus}]
Note that Lemma \ref{le:torus} justifies the LSI. By Pinsker's inequality, for any $y \in \T^d$ we have 
\begin{align*}
\big|\langle \nu - \mu_t, K (x-\cdot)\rangle \big|^2 = \big|\langle \nu - \mu_t, K (x-\cdot) - y\rangle \big|^2 \le 2\||K - y|^2\|_\infty H(\nu\,|\,\mu_t).
\end{align*}
so we may take
\begin{align*}
\gamma = 2\inf_{y \in \T^d}\||K-y|^2\|_\infty \le \frac12\sup_{x,z \in \T^d}|K(x)-K(z)|^2 =  \frac12\mathrm{diam}^2(K).
\end{align*}
Then $r_c = \frac{\sigma^4}{4\gamma \LSI } - 1$ simplifies to the claimed value.
\end{proof}

\appendix

\section{A form of Gronwall's inequality}

\begin{lemma}\label{le:Gronwall_adapted}
Let $c > 0$. Let $(g_t)_{t\geq 0}$ be a non-negative measurable  function, integrable on bounded sets. 
Let $(H_t)_{t\geq 0}$ be a non-negative measurable function such that $H_0 < \infty$. Suppose that
\begin{equation}
H_t - H_s \leq \int_s^t \left( g_u - c H_u \right) du, \quad  \text{for all } t \ge s \ge 0.
\label{pf:one_sided_entropy_estimate_Gronwall}
\end{equation}
Then 
\begin{equation}
H_t \leq e^{-ct} H_0 + \int_0^t e^{-c(t-u)} g_u du, \quad \text{for all } t \ge 0.
\label{pf:result_one_sided_entropy_estimate_Gronwall}
\end{equation}
\end{lemma}
\begin{proof}
Mollify $g$ and $H$ by a common smooth probability density of compact support. The mollifications satisfy \eqref{pf:one_sided_entropy_estimate_Gronwall}. Applying the usual differential Gronwall's inequality leads to \eqref{pf:result_one_sided_entropy_estimate_Gronwall} for the mollifications. Taking $L^1$ limits to remove the mollification leads to the original claim \eqref{pf:result_one_sided_entropy_estimate_Gronwall}, at least for \emph{almost every} $t$. But the inequality \eqref{pf:one_sided_entropy_estimate_Gronwall} implies that $\limsup_{t \to s}H_t \le H_s$ for every $s \ge 0$, and the claim is thus valid for \emph{every} $t \ge 0$.
\end{proof}

\bibliographystyle{amsplain}
\bibliography{quantsapproach-bib}

\end{document}